\documentclass[10pt, reqno]{amsart}
\usepackage[utf8]{inputenc}

\usepackage{style}
\title{Fractional coloring via entropy}

\author{Abhishek Dhawan}
\address{Department of Mathematics, University of Illinois Urbana--Champaign}
\email{adhawan2@illinois.edu}
\thanks{Partially supported by NSF RTG grant DMS-1937241.}

\date{}

\begin{document}

\begin{abstract}
    In recent work, Martinsson and Steiner proved that triangle-free $d$-degenerate graphs have fractional chromatic number $\chi_f(G) = O\left(\frac{d}{\log d}\right)$. Here, we introduce an alternate proof of the bound rooted in the analysis of the entropy of certain random variables. Beyond simplifying the original argument, our technique naturally generalizes to broader settings. In this paper, we focus on two extensions of the result.
    
    First, we consider locally $r$-colorable graphs $G$, where $\chi(G[N(v)]) \le r$ for every vertex $v$. We show that $d$-degenerate locally $r$-colorable graphs satisfy $\chi_f(G) = O\left(\frac{d\log (2r)}{\log d}\right)$, strengthening a result of Alon (1996) on the independence number of such graphs.
    
    Second, we extend Martinsson and Steiner's result to $r$-uniform $d$-degenerate hypergraphs $H$ with girth at least $4$, showing that $\chi_f(H) \le c_r\left(\frac{d}{\log d}\right)^{\frac{1}{r-1}}$ for a constant $c_r > 0$. This yields a strict generalization of a seminal result of Ajtai, Koml\'os, Pintz, Spencer, and Szemer\'edi (1982) on the independence number of uncrowded hypergraphs. Via random sampling as introduced by Duke, Lefmann, and R\"odl (1995), we obtain an improved bound on the independence number of linear hypergraphs, establishing a recent conjecture of Verstra\"ete and Wilson (2026).
    
    Our arguments generalize to the setting of fractional colorings with local demands, introduced by Kelly and Postle (2024). This yields previously unknown degree-sequence lower bounds on the independence number across each setting considered. Furthermore, our approach is constructive, yielding efficient randomized algorithms for sampling independent sets in these contexts.
\end{abstract}

\maketitle

\pagenumbering{gobble}
\tableofcontents
\newpage
\pagenumbering{arabic}

\section{Introduction}\label{section: intro}

\subsection{Background and main results}\label{subsection: background}

The fractional chromatic number of a graph $G$, denoted $\chi_f(G)$, is the minimum integer $k \ge 1$ such that there exists a probability distribution $\mathcal{D}$ over the independent sets of $G$ satisfying $\Pr_{I \sim \mathcal{D}}[v \in I] \ge 1/k$ for each vertex $v \in V(G)$.
It is well known through standard greedy arguments that
\[\chi_f(G) \le \chi(G) \le d(G) + 1 \le \Delta(G) + 1,\]
where $\Delta(G)$ and $d(G)$ denote the maximum degree and degeneracy of $G$, respectively.
Brooks's Theorem~\cite{brooks1941colouring} implies that equality holds throughout if and only if some connected component of $G$ is isomorphic to the complete graph $K_{\Delta(G) + 1}$.

A natural question arises: under what structural constraints can we obtain improved bounds on $\chi_f(\cdot)$?
A seminal result of Johansson~\cite{Joh_triangle} states that every $K_3$-free graph $G$ of maximum degree $\Delta$ has chromatic number $O(\Delta/\log\Delta)$ (see~\cite{Molloy, bernshteyn2019johansson, DKPS} for recent proofs with improved leading constants).
Shortly thereafter, Alon, Krivelevich, and Sudakov~\cite{AKSConjecture} generalized this result to \emph{locally sparse graphs}, i.e., graphs satisfying $e(G[N(v)]) \ll \Delta^2$ for all $v \in V(G)$.
They conjectured that an analogous bound holds for all $F$-free graphs, spurring extensive research over the past three decades; see, e.g.,~\cite{DKPS, anderson2024coloring, anderson2025coloring, dhawan2025bounds, AndersonBernshteynDhawan, dhawan2024palette, PS15, bonamy2022bounding, Kttt}.

Each of these results immediately implies bounds on the fractional chromatic number in terms of maximum degree.
Substantial work has also targeted bounds on $\chi_f(\cdot)$ in terms of other parameters, such as the average degree or local versions~\cite{DKPS, davies2018average, dhawan2025bounds, martinsson2025random, kelly2024fractional}.
In this paper, we focus on the degeneracy of a graph.
A graph $G$ is $d$-degenerate if every subgraph of $G$ contains a vertex of degree at most $d$.
Equivalently, there exists an ordering $\pi$ of $V(G)$ such that each vertex $v$ has at most $d$ predecessors under $\pi$.
The starting point of our investigation is the following recent result of Martinsson and Steiner~\cite{martinsson2025random}, which resolved a conjecture of Harris~\cite{harris2019some}:

\begin{theorem}[\cite{martinsson2025random}]\label{theo: martinsson steiner}
    Let $G$ be an $n$-vertex $d$-degenerate $K_3$-free graph.
    Then, $\chi_f(G) \le (4 + o(1))\frac{d}{\log d}$.
\end{theorem}

It should be noted that there exist $d$-degenerate $K_3$-free graphs $G$ satisfying $\chi(G) = d+1$~\cite{kostochka1999properties, descartes1954solution}, showing that such bounds cannot hold for ordinary chromatic number in general.
However, when $n$ is not too large, Brada\v{c}, Fox, Steiner, Sudakov, and Zhang~\cite{bradavc2026coloring} recently established that $\chi(G) = O\left(\frac{d}{\log (d/\log n)}\right)$ for $n$-vertex $d$-degenerate $K_3$-free graphs.

To prove Theorem~\ref{theo: martinsson steiner}, Martinsson and Steiner introduced a novel dynamic sampling procedure.
They assign each vertex an initial weight $w_0(v)$, traverse the graph according to a degeneracy ordering, and select each vertex $v_i$ into the independent set with probability $f(w_{i-1}(v_i)) = 1 - e^{-w_{i-1}(v_i)}$, where $w_{i-1}(u)$ denotes the updated weight of $u$ prior to step $i$.
If $v_i$ is selected, the weights of all its unexplored neighbors are set to $0$.
Conversely, if $v_i$ is omitted, the weights of its unexplored neighbors are increased to boost their likelihood of future inclusion, ensuring that the vertex weight sequence forms a martingale (i.e., $\mathbb{E}[w_t(v_i)] = w_0(v_i)$).

To establish that $v_i$ is included in the independent set with sufficiently high probability, it suffices to show that $\mathbb{E}[f(w_{i-1}(v_i))]$ is large.
However, the non-linear nature of $f(\cdot)$ renders a direct analysis intractable.
To obtain the desired lower bound, their proof tracks an auxiliary weight process governed by a modified update rule (we note that this auxiliary process is distinct for each $v_i$; i.e., they define $n$ different auxiliary processes).
Letting $\tilde{w}$ denote the modified weights for the auxiliary process corresponding to $v_i$, they demonstrate that for $K_3$-free graphs, $\mathbb{E}[f(w_{i-1}(v_i))]$ remains close to $\mathbb{E}[f(\tilde{w}_{i-1}(v_i))]$ under an appropriate choice of initial weights.
Through several key technical insights, they bound $\mathbb{E}[f(\tilde{w}_{i-1}(v_i))]$ from below, completing the proof.\footnote{We note that the actual argument differs slightly from our description here; however, these differences are unimportant for our comparison.}

In this paper, we present an alternative proof of Theorem~\ref{theo: martinsson steiner} inspired by Johansson's seminal entropy-based method for coloring $K_3$-free graphs~\cite{Joh_triangle} (see~\cite[Ch.~13]{MolloyReed} for a textbook treatment).
A key advantage of our framework lies in our choice of sampling probability: whereas Martinsson and Steiner employ the non-linear function $f(w_{i-1}(v_i)) = 1 - e^{-w_{i-1}(v_i)}$, our procedure uses a strictly linear function $g(w_{i-1}(v_i))$.
Consequently, we can compute $\mathbb{E}[g(w_{i-1}(v_i))]$ directly, bypassing auxiliary processes entirely and yielding a significantly cleaner argument.

Whether Martinsson and Steiner's approach extends to broader contexts remains unclear, as controlling their non-linear activation inherently relies on $K_3$-freeness.
In contrast, our framework extends to broader graph and hypergraph settings in a straightforward manner, demonstrating its versatility.
In this paper, we highlight two such extensions (see \S\ref{subsection: overview} for an overview of the argument and \S\ref{section: K3} for the self-contained proof of Theorem~\ref{theo: martinsson steiner}).

First, we consider \emph{locally $r$-colorable} graphs, where the induced neighborhood $G[N(v)]$ of each vertex $v \in V(G)$ is $r$-colorable (with $r = 1$ corresponding to $K_3$-free graphs).
Alon~\cite{alon1996independence} first studied the independence number $\alpha(G)$ of such graphs, showing that any $n$-vertex locally $r$-colorable graph with maximum degree $\Delta$ satisfies $\alpha(G) = \Omega\left(\frac{n}{\Delta}\frac{\log \Delta}{\log (r + 1)}\right)$.
An analogous bound for the chromatic number was later proved by Bonamy, Kelly, Nelson, and Postle~\cite{bonamy2022bounding} (see also~\cite{DKPS} for improved constants).
We show that this bound extends to the fractional chromatic number with the maximum degree $\Delta$ replaced by the degeneracy $d$, directly strengthening Alon's result.

\begin{theorem}\label{theo: r colorable ordinary coloring}
    For all $\eps > 0$ and $r \in \mathbb{N}$, there exists $d_0 \in \mathbb{N}$ such that the following holds for $d \ge d_0$ and $n \in \mathbb{N}$.
    Let $G$ be an $n$-vertex $d$-degenerate locally $r$-colorable graph.
    Then, \[\chi_f(G) \le (8+\eps)\frac{d\,\log(2r)}{\log d}.\]
    Moreover, there is a $\mathrm{poly}(n, d)$-time algorithm that samples an independent set $I \subseteq V(G)$ such that
    \[\Pr[v\in I] \ge (1-\eps)\frac{\log d}{8\,d\,\log(2r)}, \qquad \text{for each } v \in V(G).\]
\end{theorem}

Next, we turn to $r$-uniform hypergraphs $H$, where each edge contains $r$ vertices and a graph is $d$-degenerate if every subgraph has a vertex of degree at most $d$; equivalently, $H$ admits an ordering where each vertex appears last in at most $d$ edges.
The Lov\'asz local lemma implies that $r$-uniform hypergraphs of maximum degree $\Delta$ satisfy $\chi(H) = O\left(\Delta^{\frac{1}{r-1}}\right)$.
This bound fails when replacing $\Delta$ with degeneracy $d$, as Wood~\cite{wood2014hypergraph} constructed $r$-uniform $d$-degenerate hypergraphs satisfying $\chi(H) = d + 1$.
For fractional colorings, however, a simple alteration argument yields:

\begin{fact}\label{fact: hypergraph}
    Let $H$ be a $d$-degenerate $r$-uniform hypergraph. Then, $\chi_f(H) \le \left(\frac{r}{r-1}\right)\left(rd\right)^{\frac{1}{r-1}}$.
\end{fact}

\begin{proof}
    Let $(v_1, \ldots, v_n)$ be a degeneracy ordering of $V(H)$.
    Form $A \subseteq V(H)$ by including each vertex independently with probability $p \coloneqq 1/(rd)^{\frac{1}{r-1}}$.
    Let $I \subseteq A$ be the independent set obtained by removing from $A$ the last vertex of every edge $e \subseteq A$.
    For any $v \in V(H)$,
    \[\Pr[v \in I] = \Pr[v \in A,\, v \in I] = p(1 - \Pr[v\notin I\mid v \in A]).\]
    As $\Pr[v\notin I\mid v \in A] \le dp^{r-1}$, the claim follows.
\end{proof}

Extensive literature focuses on improving the local lemma bound on $\chi(H)$ under structural constraints~\cite{frieze2013coloring, cooper2015list, li2022chromatic, dhawan2025list, iliopoulos2021improved}.
Here, we initiate an analogous line of research for $\chi_f(H)$ in the setting of Fact~\ref{fact: hypergraph}.
We focus on hypergraphs of girth at least $4$.
Ajtai, Koml\'os, Pintz, Spencer, and Szemer\'edi~\cite{ajtai1982extremal} first established that uncrowded (girth $\ge 5$) hypergraphs contain independent sets of size at least $c_r\,n \left(\frac{\log \Delta}{\Delta}\right)^{\frac{1}{r-1}}$.
Frieze and Mubayi~\cite{frieze1990independence} adapted Johansson's entropy method to show that $r$-uniform hypergraphs with maximum degree $\Delta$ and girth at least $4$ satisfy $\chi(H) \le c_r'\left(\frac{\Delta}{\log \Delta}\right)^{\frac{1}{r-1}}$.
We prove that this bound holds for fractional chromatic number with maximum degree replaced by degeneracy:

\begin{theorem}\label{theo: hypergraph free ordinary coloring}
    For all $\eps > 0$ and integer $r \ge 2$, there exists $d_0 \in \mathbb{N}$ such that the following holds for $d \ge d_0$ and $n \in \mathbb{N}$.
    Let $H$ be an $n$-vertex $r$-uniform $d$-degenerate hypergraph with girth at least $4$. Then,
    \[\chi_f(H) \le \left(1+\eps\right)\left(\frac{r}{r-1}\right)\left(r(r-1)^2\frac{d}{\log d}\right)^{\frac{1}{r-1}}.\]
    Moreover, there is a $\mathrm{poly}(n, d)$-time algorithm that samples an independent set $I \subseteq V(H)$ such that
    \[\Pr[v\in I] \ge \left(1-\eps\right)\left(1 - \frac{1}{r}\right)\left(\frac{\log d}{r(r-1)^2d}\right)^{\frac{1}{r-1}}, \qquad \text{for each } v \in V(H).\]
\end{theorem}

Note that setting $r = 2$ in Theorem~\ref{theo: hypergraph free ordinary coloring} recovers the bound in Theorem~\ref{theo: martinsson steiner}, thereby yielding a substantial generalization of it.
Our proof of Theorem~\ref{theo: hypergraph free ordinary coloring} is inspired by the approach of Frieze and Mubayi~\cite{frieze1990independence}, who combine entropy techniques with the R\"odl nibble method.
However, unlike the standard nibble framework that processes vertices in batches, our algorithm processes vertices individually along the degeneracy ordering.
This operational distinction necessitates several key modifications that form the core technical novelty of our approach (we discuss this further in~\S\ref{subsection: overview}).



We remark that Theorem~\ref{theo: hypergraph free ordinary coloring} immediately yields a lower bound on the independence number for girth-$4$ hypergraphs in terms of degeneracy.
Recently, Verstra\"ete and Wilson posed the following conjecture regarding the independence number of linear hypergraphs:

\begin{conjecture}[{\cite[Conjecture~1]{verstraete2026independent}}]\label{conj: linear vw}
    For each $r \ge 3$, there exists a constant $a_r > 0$ satisfying $a_r = 1 - o_r(1)$ as $r \to \infty$ such that every $n$-vertex $r$-uniform linear hypergraph $H$ with maximum degree $\Delta$ satisfies
    \[\alpha(H) \ge (a_r - o_\Delta(1))\,n\,\left(\frac{\log \Delta}{\Delta}\right)^{\frac{1}{r-1}}.\]
\end{conjecture}

The above would be best possible as there exist random constructions attaining the bound $\alpha(H) = (a_r \pm o_\Delta(1))\,n\,\left(\frac{\log \Delta}{\Delta}\right)^{\frac{1}{r-1}}$ with $a_r = \left(\frac{r}{r-1}\right)^{\frac{1}{r-1}}$ (see, e.g., \cite[eq.~18]{verstraete2026independent}).
While Conjecture~\ref{conj: linear vw} was previously known to hold under the stronger assumption that $H$ has girth $\omega(\Delta)$~\cite{nie2021randomized}, the best-known general bound achieves only
\[a_r = \min\left\{\frac{1}{r-1}\left(\frac{r-2}{4r - 5}\right)^{\frac{1}{r-1}},\, \frac{0.98}{e}\left(\frac{(r-2)}{10^5(r-1)(4r - 5)}\right)^{\frac{1}{r-1}}\right\},\]
which gives $a_r \to \frac{0.98}{e}$ as $r \to \infty$.
This earlier constant is obtained by combining Iliopoulos's bound on the chromatic number of uncrowded hypergraphs~\cite{iliopoulos2021improved} and Ajtai, Koml\'os, Pintz, Spencer, and Szemer\'edi's bound on the independence number~\cite{ajtai1982extremal}, with the random sampling framework of Duke, Lefmann, and R\"odl~\cite{duke1995uncrowded} (cf.~\cite[p.~2]{verstraete2026independent}).
As $d(H) \leq \Delta(H)$, combining Theorem~\ref{theo: hypergraph free ordinary coloring} with the aforementioned random sampling framework confirms Conjecture~\ref{conj: linear vw} for all linear hypergraphs.

\begin{corollary}
    For all $r \ge 3$, there exists $\Delta_0 \in \mathbb{N}$ such that the following holds for $\Delta \ge \Delta_0$ and $n \in \mathbb{N}$.
    Let $H$ be an $n$-vertex $r$-uniform linear hypergraph with maximum degree $\Delta$.
    Then, \[\alpha(H) \ge (a_r - o_\Delta(1))\,n\,\left(\frac{\log \Delta}{\Delta}\right)^{\frac{1}{r-1}},\]
    where $a_r = \left(1 - \frac{1}{r}\right)\left(\frac{r - 2}{r(r-1)^2(3r - 4)}\right)^{\frac{1}{r-1}}$.
\end{corollary}

Furthermore, in concurrent work, Methuku, Vo, and the author~\cite{triangle_free} established Conjecture~\ref{conj: linear vw} with a sharper constant, namely $a_r = \left(\frac{r - 2}{(r-1)(4r-5)}\right)^{\frac{1}{r-1}}$.
By leveraging techniques developed by Janzer, Methuku, and the author~\cite{Kttt}, that work also improves Iliopoulos's bound on the independence number of uncrowded hypergraphs~\cite{iliopoulos2021improved}, thereby advancing the classical result of Ajtai, Koml\'os, Pintz, Spencer, and Szemer\'edi~\cite{ajtai1982extremal}.

\subsection{Fractional coloring with local demands}\label{sec: main results}

Kelly and Postle~\cite{kelly2024fractional} recently introduced the theory of \textit{fractional coloring with local demands} wherein each vertex $v$ ``demands'' a certain probability of being included in the random independent set $I$.
This demand depends on the local neighborhood structure surrounding $v$. 
To extend this framework, we generalize \cite[Definition~1.2]{kelly2024fractional} to hypergraphs as follows:

\begin{definition}[$\phi$-colorings]
    Given a hypergraph $H$, a \textit{demand function} for $H$ is a function $\phi\,:\, V(H) \to [0, 1]$.
    A $\phi$-coloring of $H$ is a probability distribution $\mathcal{D}$ over the independent sets of $H$ such that $\Pr_{I \sim \mathcal{D}}[v\in I] \geq \phi(v)$ for each vertex $v \in V(H)$.
\end{definition}

This definition enables sampling larger independent sets in graphs with non-uniform degree sequences. For example, a vertex of minimum degree $\delta(G)$ is a better candidate for $I$ than one of degree $\Delta(G)$, as lower-degree vertices rule out fewer remaining candidates. Kelly and Postle established several fundamental results in their paper, including a local fractional greedy bound that generalizes the Caro–Wei theorem~\cite{caro1979new}, a local fractional version of Vizing's theorem for multigraphs~\cite{Vizing}, and progress toward a local fractional version of Shearer's celebrated bound on the independence number of $K_3$-free graphs~\cite{shearer1991note}. Recently, Martinsson and Steiner resolved a conjecture of Kelly and Postle by proving this local fractional version of Shearer's bound:

\begin{theorem}[\cite{martinsson2025random}]\label{theo: martinsson steiner local demands}
    Let $G$ be a $K_3$-free graph and let $\phi$ be a demand function for $G$.
    Then, $G$ admits a $\phi$-coloring if
    \[\phi(v) \leq (1-o(1))\frac{\log \deg(v)}{\deg(v)}, \qquad \text{for each } v \in V(G).\]
\end{theorem}

Theorem~\ref{theo: martinsson steiner local demands} is proved by an intricate modification of Shearer-style induction.
However, extending this argument to other settings appears difficult, as Shearer-style induction is notoriously rigid with respect to the $K_3$-freeness assumption.
In contrast, our framework provides a flexible approach for tackling fractional colorings of graphs and hypergraphs.
Indeed, a simple adaptation of our arguments for Theorems~\ref{theo: r colorable ordinary coloring} and~\ref{theo: hypergraph free ordinary coloring} yields the corresponding local bounds---specifically by processing vertices in increasing order of demand rather than in a degeneracy ordering.

\begin{theorem}\label{theo: local demands r colorable}
    For all $\eps > 0$ and $r \in \N$ there exists $\delta_0 \in \N$ such that the following holds for $\delta \geq \delta_0$ and $n \in \N$.
    Let $G$ be an $n$-vertex locally $r$-colorable graph such that $\delta(G) \geq \delta$ and let $\phi$ be a demand function for $G$.
    Then, $G$ admits a $\phi$-coloring if
    \[\phi(v) \leq (1-\eps)\dfrac{\log \deg(v)}{8\,\deg(v)\,\log(2r)}, \qquad \text{for each } v \in V(G).\]
    Moreover, there is a $\poly(n)$-time algorithm that samples an independent set $I \subseteq V(G)$ such that
    \[\Pr[v\in I] \geq \phi(v), \qquad \text{for each } v \in V(G).\]
\end{theorem}

\begin{theorem}\label{theo: hypergraph local demands}
    For all $\eps > 0$ and integer $r \geq 2$ there exists $\delta_0 \in \N$ such that the following holds for $\delta \geq \delta_0$ and $n \in \N$.
    Let $H$ be an $n$-vertex $r$-uniform hypergraph having girth at least $4$ such that $\delta(G) \geq \delta$ and let $\phi$ be a demand function for $H$.
    Then, $H$ admits a $\phi$-coloring if
    \[\phi(v) \leq \left(1-\eps\right)\left(1 - \frac{1}{r}\right)\left(\dfrac{\log \deg(v)}{r(r-1)^2\deg(v)}\right)^{\frac{1}{r-1}}, \qquad \text{for each } v \in V(H).\]
    Moreover, there is a $\poly(n)$-time algorithm that samples an independent set $I \subseteq V(H)$ such that
    \[\Pr[v\in I] \geq \phi(v), \qquad \text{for each } v \in V(H).\]
\end{theorem}

We remark that Theorem~\ref{theo: local demands r colorable} implies that 
\begin{align}\label{eq: local independence r colorable}
    \alpha(G) \geq \frac{(1-\eps)}{8\log(2r)}\sum_{v\in V(G)}\dfrac{\log \deg(v)}{\deg(v)}
\end{align}
for locally $r$-colorable graphs $G$.
Prior to our work, such a bound was known to hold only under stronger minimum degree assumptions (e.g., $\delta(G) \ge \mathsf{polylog}(\Delta(G))$~\cite{DKPS, bonamy2022bounding}). 
In contrast, our result requires only $\delta(G) = \Omega(1)$, which is substantially milder.

Similarly, Theorem~\ref{theo: hypergraph local demands} implies that 
\begin{align}\label{eq: local independence hypergraph}
    \alpha(H) \geq \left(1-\eps\right)\left(1 - \frac{1}{r}\right)\frac{1}{(r(r-1)^2)^{\frac{1}{r-1}}}\sum_{v\in V(H)}\left(\dfrac{\log \deg(v)}{\deg(v)}\right)^{\frac{1}{r-1}}
\end{align}
for hypergraphs $H$ having girth at least $4$.
Such a bound was unknown prior to our work for any regime of $\delta(H)$.

We conclude this section with a short proof of a local fractional version of Fact~\ref{fact: hypergraph}, which generalizes the analogous independence number bound for $r$-uniform hypergraphs due to Caro and Tuza~\cite{caro1991improved}, albeit with a weaker dependence on $r$ (see also~\cite{spencer1972turan, dutta2012new}).

\begin{theorem}\label{thm: local fractional caro tuza}
    The following holds for integer $r \ge 2$.
    Let $H$ be an $r$-uniform hypergraph with no isolated vertices, and let $\phi$ be a demand function for $H$.
    Then, $H$ admits a $\phi$-coloring if
    \[\phi(v) \leq \left(1 - \frac{1}{r}\right) \frac{1}{(r \deg(v))^{\frac{1}{r-1}}}, \qquad \text{for each } v \in V(H).\]
    In particular, the independence number of $H$ satisfies
    \[\alpha(H) \ge \left(1 - \frac{1}{r}\right) \sum_{v \in V(H)} \frac{1}{(r \deg(v))^{\frac{1}{r-1}}}.\]
\end{theorem}

\begin{proof}
    Let $(v_1, \ldots, v_n)$ be an ordering of $V(H)$ in increasing order of demands.
    Form $A \subseteq V(H)$ by including each vertex $v_i$ independently with probability $p_i \coloneqq \frac{r}{r-1}\phi(v_i) \leq 1$.
    Let $I \subseteq A$ be the independent set obtained by removing from $A$ the last vertex (under our ordering) of every edge $e \subseteq A$.
    For a vertex $v_i \in V(H)$, we have
    \begin{align}\label{eq: local fractional deg}
        \Pr[v_i \in I] = \Pr[v_i \in A,\, v_i \in I] = p_i(1 - \Pr[v_i\notin I\mid v_i \in A]).
    \end{align}
    Note that
    \begin{align*}
        \Pr[v_i \notin I \mid v_i \in A] \leq \sum_{\substack{e \ni v_i \\ v_i = \max_{w \in e}w}}\Pr[e \subseteq A \mid v_i \in A] 
        &= \sum_{\substack{e \ni v_i \\ v_i = \max_{w \in e}w}}\prod_{u \in e \setminus \{v_i\}}\left(\frac{r}{r-1}\phi(u)\right) \\
        &\leq \sum_{\substack{e \ni v_i \\ v_i = \max_{w \in e}w}}\prod_{u \in e \setminus \{v_i\}}\left(\frac{r}{r-1}\phi(v_i)\right) \\
        &\leq \sum_{\substack{e \ni v_i \\ v_i = \max_{w \in e}w}}\frac{1}{r \deg(v)} \leq \frac{1}{r},
    \end{align*}
    where we use the fact that $\phi(u) \leq \phi(v_i)$ as a consequence of our vertex ordering.
    Plugging the above into \eqref{eq: local fractional deg} completes the proof.
\end{proof}

\subsection{Notation and terminology}\label{subsection: notation}

Throughout the rest of the paper we use the following basic notation. For $n \in \N$, we let $[n] \defeq \set{1, \ldots, n}$.
For a hypergraph $H$, its vertex and edge sets are denoted $V(H)$ and $E(H)$, respectively.
For a vertex $v \in V(H)$, $\deg_H(v) \coloneqq |\set{e \in E(H)\,:\, v \in e}|$ denotes the degree of $v$; we drop the subscript $H$ when the context is clear.
We say $u \neq v$ is a neighbor of $v$ if $\set{u , v} \subseteq e$ for some edge $e\in E(H)$.
For a subset $U \subseteq V(H)$, the subhypergraph induced by $U$ is denoted by $H[U]$.

We say an $n$-vertex hypergraph $H$ is \emphd{$d$-degenerate} for $d\in \N$ if there exists an ordering $(v_1, \ldots, v_n)$ of $V(H)$ such that for each $i\in [n]$, we have $\deg_{H_i}(v_i) \leq d$, where $H_i \coloneqq H[\set{v_1, \ldots, v_i}]$.
Given an ordering of $V(H)$, we let $N_L(v_i)$ be the set of neighbors of $v_i$ in the hypergraph $H_i$; we call such vertices \emph{left-neighbors} of $v_i$.
Similarly, we let $N_R(v_i)$ be the vertices $v_k$ such that $v_i \in N_L(v_k)$, i.e., the \emph{right-neighbors} of $v_i$.
We let $d_L(v_i) = |N_L(v_i)|$ and $d_R(v_i) = |N_R(v_i)|$ denote the left- and right-degrees of $v_i$ respectively.
Finally, given an edge $e$ and a vertex $u \in e$, we say that $u$ is an \emph{internal vertex} of $e$ if $u\in N_L(v)$ for some $v \in e$, i.e., $u$ is not the right-most vertex in $e$ under the given degeneracy ordering.

\subsection{Proof Overview}\label{subsection: overview}

In this section, we provide an informal overview of our proof techniques.
We first discuss the overarching algorithmic template for $d$-degenerate hypergraphs and then provide details on the adaptation to each of our main results.

The algorithm takes as input a $d$-degenerate (hyper)graph $H$ and outputs an independent set $I \subseteq V(H)$.
We initially assign each vertex $v$ a weight $p_0(v) = \alpha$ for a well-chosen $\alpha \in [0, 1]$.
Iterating through the vertices in a degeneracy ordering of $H$, we include the vertex $v_i$ in $I$ with a probability roughly $\Theta(p_{i-1}(v_i))$ (we are being intentionally vague at this point as there is a distinction in how we do so for each of our results).
We update the weights $p_i(v_k)$ for each $k > i$ according to the outcome of the event $\set{v_i\in I}$.
In particular, the weight $p_i(v_k)$ must be set to $0$ if some edge $e\ni v_k$ satisfies $(e\setminus \set{v_k}) \subseteq I$.
Additionally, we increase the weight for $v_k \in N_R(v_i)$ if $v_i\notin I$ as we would like to make it more likely that $v_k \in I$ in this case.
Our choice of parameters ensures that $\E[p_i(v_k)\mid p_{i-1}(\cdot)] = p_{i-1}(v_k)$.
It then follows that $\Pr[v_i \in I] = \Theta(\alpha)$.

There is, however, a technical fiddle: what if $p_i(v)$ becomes too large (e.g., $p_i(v) > 1$)?
As the weights represent a probability, we cannot allow this to occur.
For a well-chosen $\hat p$, we artificially set $p_i(v) = \hat p$ for vertices $v$ satisfying $p_i(v) > \hat p$ and label them as \emph{bad}.
Bad vertices are never added to $I$.
As a result, we no longer have the nice property $\E[p_i(v_k)\mid p_{i-1}(\cdot)] = p_{i-1}(v_k)$.
A standard trick, often referred to as an \emph{equalizing coin flip}, allows us to achieve this property in the case that $p_i(v) > \hat p$.
See~\cite[Ch.~13]{MolloyReed} for a more in-depth discussion of the necessity of such thresholding in the context of coloring $K_3$-free graphs.

At this point, we can show the following:
\[\Pr[v_k \in I] = \Theta\left(\alpha - \hat p \Pr[v_k \text{ is bad}]\right).\]
To show $\Pr[v_k \in I]$ is ``large'', we must show that $\Pr[v_k \text{ bad}]$ is ``small.''
This is where our entropy argument appears.
In general, given a probability distribution $\mathcal{D}$ over some finite set $T$, the entropy of $\mathcal{D}$ is the quantity
\[H(\mathcal{D}) = -\sum_{t\in T}\Pr_{X \sim\mathcal{D}}[X = t]\log \Pr_{X \sim\mathcal{D}}[X = t].\]
Roughly speaking, $H(\mathcal{D})$ measures how widely the assignment probabilities vary---the more they vary, the smaller their entropy will be; the above is maximized when $\mathcal{D}$ is the uniform distribution.
We define the \textit{entropy} at $v$ after processing $i$ vertices to be
\[H_i(v) = -p_i(v)\log p_i(v).\]
Using the intuition from probability theory, if $\mathbb{E}[H_{k-1}(v_k)]$ is bounded away from zero, then the quantity $p_{k-1}(v_k)$ does not exhibit large fluctuations.
Since $\mathbb{E}[p_i(v)] = \alpha$, it follows that $p_{k-1}(v)$ is unlikely to reach a value as large as $\hat{p}$.
We establish that $\mathbb{E}[H_{k-1}(v_k)]$ is indeed sufficiently large, thereby confirming that $\Pr[v_k \text{ is bad}]$ is suitably small, as required.

In the remainder of this overview, we provide more details on our sampling and weight-update procedures for each of our main results.

\subsubsection*{Triangle-free graphs}

Here, we select $v_i$ in $I$ with probability $p_{i-1}(v_i)$.
We update the weights $p_i(v_k)$ for each $k > i$ according to the outcome of the event $\set{v_i \in I}$.
More formally, let $v_iv_k \in E(G)$ be such that $k > i$.
We update $p_i(v_k)$ as follows:
\begin{align}\label{eq: update intro K3}
    p_i(v_k) = \left\{\begin{array}{cc}
    0 & \text{if } v \in I; \\[1em]
    \dfrac{p_{i-1}(v_k)}{1 - p_{i-1}(v_i)} & \text{if } v_i \notin I.
\end{array}\right.
\end{align}
The above rule clearly ensures that $\E[p_i(v_k) \mid p_{i - 1}(\cdot)] = p_{i-1}(v_k)$.
For our proof of Theorem~\ref{theo: martinsson steiner}, we use the threshold $\hat p = \alpha^{o(1)}$ and define equalizing coin flips accordingly.
See Algorithm~\ref{algorithm: fcp k3} for a formal description of the procedure.

\subsubsection*{Locally $r$-colorable graphs $G$}

We select a vertex $v_i$ to be in $I$ via a two-step process: first, we \textit{activate} $v_i$ with probability $p_{i-1}(v_i)$; then, conditional on activation, we \textit{select} $v_i$ to enter $I$ with probability $1/2$.
We subsequently update the weights $p_i(v_k)$ for all $v_k \in V(G)$ based on the outcome of the event that $v_i \in I$.
Specifically, if $v_i \in I$, we set $p_i(v_k) = 0$ for all $v_k \in N_R(v_i)$, thereby ensuring that $I$ remains an independent set.
If $v_i$ is not activated, the weights remain unchanged.
However, if $v_i$ is activated but not selected, we sample a color class $J \subseteq N_R(v_i)$ uniformly at random from a proper $r$-coloring of $G[N_R(v_i)]$.
We then increase the weights of all vertices $v_k \in J$ by a factor of $2r$ (boosting their likelihood of future selection into $I$) and set the weights of all remaining vertices $v_j \in N_R(v_i) \setminus J$ to $0$.

To briefly motivate this update rule, one can view the activation of $v_i$ as identifying it as a strong candidate for inclusion in $I$.
If $v_i$ is activated but ultimately rejected, we seek to include a large subset of $N_R(v_i)$ in $I$ to act as a ``certificate'' for the omission of $v_i$ from $I$.
Because this certificate must form an independent set, we increase the weights of vertices contained within a large independent subset of $N_R(v_i)$.\footnote{We remark that this argument holds, with minor modifications, even if the graph is only locally fractionally $r$-colorable rather than properly locally $r$-colorable.}

For the proof of Theorem~\ref{theo: r colorable ordinary coloring}, we use the threshold $\hat p = 1$ and define equalizing coin flips accordingly.
See Algorithm~\ref{algorithm: fcp} for a formal description of the procedure.

\subsubsection*{High-girth hypergraphs $H$}

Here, we select $v_i$ into $I$ with probability $p_{i-1}(v_i)$.
We then update the weights $p_i(v_k)$ for each $k > i$ based on the outcome of the event $\{v_i \in I\}$.
More formally, let $e$ be an edge containing $\{v_i, v_k\}$ such that $v_k$ is the right-most vertex in $e$ under the ordering.
(Note that $e$ is unique since $H$ is linear.
Furthermore, edges in which $v_k$ appears as the right-most vertex are the only ones capable of ``killing'' $v_k$, i.e., forcing $p_{k-1}(v_k) = 0$.)
Additionally, let $f \subseteq e$ consist of those vertices in $e$ that precede $v_i$ in the degeneracy ordering (see Figure~\ref{fig: edge} for an illustration).
We update $p_i(v_k)$ according to the following rules:
\begin{align}\label{eq: update intro}
    p_i(v_k) = \left\{\begin{array}{cc}
    p_{i-1}(v_k) & \text{if } u \notin I \text{ for some } u \in f;  \\[2em]
    \dfrac{p_{i-1}(v_k)}{1 - \prod_{u \in e \setminus (f \cup \{v_k\})}p_{i-1}(u)} & \text{if } v_i \notin I; \\[2em]
    p_{i-1}(v_k)\left(\dfrac{1 - \prod_{u \in e \setminus (f \cup \{v_i, v_k\})}p_{i-1}(u)}{1 - \prod_{u \in e \setminus (f \cup \{v_k\})}p_{i-1}(u)}\right) & \text{otherwise}.
\end{array}\right.
\end{align}
Note that the last case implies that if $u \in I$ for every $u \in e \setminus \{v_k\}$, then $p_{k-1}(v_k) = 0$; i.e., $v_k$ is not included in $I$.

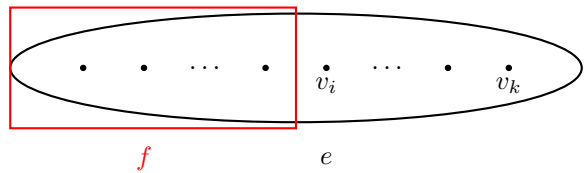
\begin{figure}[htb!]
    \centering
    \begin{tikzpicture}[scale=0.8]
        \node[circle,fill=black,draw,inner sep=0pt,minimum size=2pt] (v1) at (0, 0) {};
        \node[circle,fill=black,draw,inner sep=0pt,minimum size=2pt] (v2) at (1, 0) {};
        \node[circle,fill=black,draw,inner sep=0pt,minimum size=2pt] (v3) at (3, 0) {};
        \node[circle,fill=black,draw,inner sep=0pt,minimum size=2pt] (v4) at (4, 0) {};
        \node[circle,fill=black,draw,inner sep=0pt,minimum size=2pt] (v5) at (6, 0) {};
        \node[circle,fill=black,draw,inner sep=0pt,minimum size=2pt] (v6) at (7, 0) {};

        \node at (2, 0) {$\ldots$};
        \node at (5, 0) {$\ldots$};

        \node[anchor=north] at (v4) {$v_i$};
        \node[anchor=north] at (v6) {$v_k$};
        \node[] at (4, -1.5) {$e$};
        \node[red] at (1, -1.5) {$f$};
        
        \draw[thick] (3.5,0) ellipse (4.7cm and 0.9cm);
        \draw[thick, red] (-1.2,-1) rectangle (3.5,1);
        
    \end{tikzpicture}
    \caption{An edge $e$ containing $v_i$ and $v_k$ such that $v_k$ is the right-most vertex of $e$. Here, $f$ consists of the vertices of $e$ appearing before $v_i$ in the degeneracy ordering.}
    \label{fig: edge}
\end{figure}

Let us motivate this update rule by first recalling the approach of Frieze and Mubayi~\cite{frieze2013coloring}, who proved a bound on the chromatic number of hypergraphs with girth at least $4$.
In their work, they employ a variant of the R\"odl nibble method (see~\cite{KangKelly2023nibble} for a survey on applications to hypergraph coloring).
Crucially, the nibble method proceeds in stages: during stage~$i$, each uncolored vertex $v$ samples a subset of colors $S_i(v) \subseteq [q]$ by including each color $c$ independently with probability $\eps p_{i-1}(v, c)$. 
Weights are then updated accordingly—namely, if a color $c$ is selected by every vertex $u \in e \setminus \{v\}$ for some edge $e \ni v$, then $c$ becomes unviable for $v$, forcing $p_i(v, c) = 0$.
A vertex is successfully colored once it selects a non-empty subset.\footnote{The precise implementation differs slightly from this simplified description, but these differences are immaterial for our overview.} 
Because the underlying hypergraph is linear, the probability that a color $c$ is not removed at $v$ due to a specific edge $e \ni v$ during stage~$i$ is given by
\[q_{e, i}(v, c) \coloneqq 1 - \prod_{u \in e \setminus \{v\}} \big(\eps p_{i-1}(u, c)\big).\]
To ensure that the expected weight of $p_i(v, c)$ remains invariant, Frieze and Mubayi set the update rule to
\[p_i(v, c) = \frac{p_{i-1}(v, c)}{\prod_{e \ni v} q_{e, i}(v, c)}.\]
While natural, the analogous vertex weight update rule cannot be applied in our setting because we process vertices one at a time rather than in parallel batches. 
In the $i$-th stage of the nibble method, every remaining uncolored vertex makes random choices simultaneously. 
In contrast, during the $i$-th step of our algorithm, only a single vertex makes random choices. 
Furthermore, our algorithm processes each vertex exactly once, whereas the nibble method repeatedly visits uncolored vertices until a color is successfully assigned. These fundamental structural differences necessitate a new update procedure.

Let us now return our attention to the setting of our update rule.
I.e., we are processing a vertex $v_i$, and $v_k\in N_R(v_i)$ is such that $e \ni \set{v_i, v_k}$.
Additionally, let $f \subseteq e$ consist of the vertices of $e$ that appear before $v_i$ in the degeneracy ordering.
The probability that $v_k$ is killed due to $e$ in the algorithm is roughly $q_{e}(v_k) \coloneqq 1 - \prod_{v_j \in e\setminus \set{v_k}}p_{j-1}(v_j)$.
As we cannot determine $q_{e}(v_k)$ a-priori, our update rule enables a semblance of a ``running computation'' of $q_{e}(v_k)$.
Indeed, one can view the rule as related to the conditional probability that $e$ does not kill $v_k$ given the outcomes of vertices in $e$ processed thus far; i.e., we have
\[\Pr[e \text{ does not kill } v_k \mid f \subseteq I] \approx 1 - \prod_{u \in e \setminus (f \cup \{v_k\})}p_{i-1}(u),\]
and
\[\Pr[e \text{ does not kill } v_k \mid f\cup \{v_i\} \subseteq I] \approx 1 - \prod_{u \in e \setminus (f \cup \{v_i,v_k\})}p_{i-1}(u).\]
Additionally, we have
\[\Pr[e \text{ does not kill } v_k \mid f \not \subseteq I] = 1,\]
covering all possible cases.
It follows that the net effect of the edge $e$ on the weight of $v_k$ is roughly
\[\prod_{v_i \in e \setminus \{v_k\}}\frac{\Pr[e \text{ does not kill } v_k \mid I \cap \{v_1, \ldots, v_i\}]}{\Pr[e \text{ does not kill } v_k \mid I \cap \{v_1, \ldots, v_{i-1}\}]} \approx \frac{1}{\Pr[e \text{ does not kill } v_k]} = \frac{1}{q_e(v_k)},\]
as the product is telescoping in expectation.

For the proof of Theorem~\ref{theo: hypergraph free ordinary coloring}, we use the threshold $\hat p = \alpha^{o(1)}$.
A key step in the proof of Theorems~\ref{theo: martinsson steiner} and~\ref{theo: r colorable ordinary coloring} relies on the fact that a weight either increases, remains the same, or is set to $0$ during each iteration.
We no longer have this property as a result of the final case in the update process~\eqref{eq: update intro}.
Therefore, we need an additional threshold step that ensures a weight does not drop too much.
We ensure all non-zero weights are at least $\alpha^{1+o(1)}$, and define additional equalizing coin flips accordingly.
See Algorithm~\ref{algorithm: fcp hypergraph} for a formal description of the procedure.

\subsubsection*{Local demands}
The proofs of Theorems~\ref{theo: local demands r colorable} and~\ref{theo: hypergraph local demands} follow the same general template. However, rather than processing vertices according to a degeneracy ordering, we process them in increasing order of demand, noting that each vertex $v$ has at most $(r-1)\deg_H(v)$ vertices appearing before it in this ordering.

\subsection*{Structure of the paper}
The remainder of the paper is structured as follows: in \S\ref{section: K3}, we provide our alternate proof of Theorem~\ref{theo: martinsson steiner}; in \S\ref{section: locally colorable}, we prove Theorems~\ref{theo: r colorable ordinary coloring} and~\ref{theo: local demands r colorable}; in \S\ref{section: hypergraph}, we prove Theorems~\ref{theo: hypergraph free ordinary coloring} and~\ref{theo: hypergraph local demands}; finally, in \S\ref{section: concluding remarks}, we outline potential future avenues of work.

\section{$K_3$-free graphs: an alternate proof of Theorem~\ref{theo: martinsson steiner}}\label{section: K3}

In this section, we will prove the following result, which bounds the fractional chromatic number of $d$-degenerate $K_3$-free graphs, providing an alternate proof of Theorem~\ref{theo: martinsson steiner}.

\begin{theorem}\label{theo: k3}
    For all $\eps > 0$ there exists $d_0 \in \N$ such that the following holds for $d \geq d_0$ and $n \in \N$.
    Let $G$ be an $n$-vertex $d$-degenerate $K_3$-free graph.
    Then, \[\chi_f(G) \leq (4+\eps)\dfrac{d}{\log d}.\]
    Moreover, there is a $\poly(n, d)$-time algorithm that samples an independent set $I \subseteq V(G)$ such that
    \[\Pr[v\in I] \geq (1-\eps)\dfrac{\log d}{4\,d}, \qquad \text{for each } v \in V(G).\]
\end{theorem}

Let us describe our independent set procedure (see Algorithm~\ref{algorithm: fcp}).
The algorithm takes as input an $n$-vertex $d$-degenerate $K_3$-free graph $G$, a degeneracy ordering $(v_1, \ldots, v_n)$ of the vertices $V(G)$, and a parameter $\eps > 0$, and outputs an independent set in $G$.
We refer the reader to \S\ref{subsection: overview} for an informal overview of the procedure.
Let us now provide a formal description of the algorithm.

\vspace{10pt}
\begin{breakablealgorithm}
\caption{Random Independent Set Procedure for $K_3$-Free Graphs}\label{algorithm: fcp k3}
\begin{flushleft}
\textbf{Input}: An $n$-vertex $d$-degenerate $K_3$-free graph $G$, a degeneracy ordering $(v_1, \ldots, v_n)$ of the vertices $V(G)$, and a parameter $\eps > 0$. \\
\textbf{Output}: An independent set in $G$.
\end{flushleft}

\begin{enumerate}[itemsep=5pt]
    \item \textbf{Initialize:} Set $I = B_0 = \emptyset$.
    For each $v \in V(G)$, set $p_0(v) = \alpha$, where $\alpha = \frac{\log d}{(2 + \eps/2)d}$, and let $\hat p = \alpha^{\eps/10}$.

    \item \textbf{Iterate:} For $i = 1, \ldots, n$, do the following:
            \begin{enumerate}[label=\ep{\roman*}, itemsep=5pt]
                \item For each $j > i$ such that $v_j \notin N_R(v_i)$, set $p_i(v_j) = p_{i-1}(v_j)$.
                \item \label{step: ai} Let $a_{i} \sim \mathrm{BER}(p_{i-1}(v_i))$. If $a_{i} = 1$, then add $v$ to $I$.
                \item For each $v_j \in N_R(v_i)$, do the following:
                \begin{enumerate}[label=\ep{\roman*}, itemsep=5pt]
                    \item\label{step: ai delete k3} If $a_{i} = 1$, set $p_i(v_j) = \hat p$ with probability $\mu_{i, j} \coloneqq \max\left\{0,\, \frac{p_{i-1}(v_j)/\hat p - (1 - p_{i-1}(v_i))}{p_{i-1}(v_i)}\right\}$ and $0$ otherwise.
                    \item If $a_{i} = 0$, set $p_i(v_j) = \min\left\{\frac{p_{i-1}(v_j)}{1 - p_{i-1}(v_i)},\, \hat p\right\}$.
                \end{enumerate}
            \item\label{step: bad update k3} Set $B_i = \set{v_j\,:\, p_i(v_j) = \hat p,\, j > i}$.
        \end{enumerate}
\end{enumerate}
\end{breakablealgorithm}
\vspace{10pt}

Note that Algorithm~\ref{algorithm: fcp k3} can be implemented in $O(nd)$ time satisfying the efficiency requirements and performance guarantees outlined in Theorem~\ref{theo: k3}.
As a result of Steps~\ref{step: ai delete k3}~and~\ref{step: bad update k3}, the set $I$ is an independent set.
Additionally, the probability $\mu_{i, j}$ is well-defined: indeed, $\mu_{i, j} \leq 1$ since $p_{i-1}(v_j) \leq \hat p$. We have $\mu_{i, j} > 0$ if and only if $p_{i-1}(v_j)/(1 - p_{i-1}(v_i)) > \hat p$.
It remains to show that $\Pr[v_k \in I]$ is large for each vertex $v_k$.
To this end, note the following:
\begin{align}\label{eq: K3 main bound}
    \Pr[v_k \in I] = \E[p_{k-1}(v_k)\mathbf{1}\set{v_k \notin B_{k-1}}] = \E[p_{k-1}(v_k)] - \hat p \Pr[v_k \in B_{k-1}].
\end{align}

Let us begin by showing that the random variable $p_i(v_k)$ is a martingale in $i$, which in particular implies $\E[p_{k-1}(v_k)] = \alpha$.

\begin{lemma}\label{lemma: K3 martingale}
    For $0 \leq i < k \leq n$, we have $\E[p_i(v_k)] = \alpha$.
\end{lemma}

\begin{proof}\stepcounter{ForClaims} \renewcommand{\theForClaims}{\ref{lemma: K3 martingale}}
    The following claim shows that $p_i(v_k)$ is a martingale in $i$ for $i = 0, \ldots, k-1$.
    \begin{claim}\label{claim: pi k3}
        $\E[p_i(v_k) \mid p_{i-1}(\cdot)] = p_{i-1}(v_k)$ for $i = 1, \ldots, k-1$.
    \end{claim}
    
    \begin{claimproof}
        If $v_i \notin N_L(v_k)$ or $\{v_i, v_k\} \cap B_{i-1} \neq \emptyset$, the claim is trivial.
        Suppose $v_i \in N_L(v_k)$.
        We have two cases to consider.
        \begin{itemize}[itemsep=5pt]
            \item \textbf{Case 1:} $\frac{p_{i-1}(v_k)}{1 - p_{i-1}(v_i)} \leq \hat p$.
            We have
            \begin{align*}
                \E[p_i(v_k) \mid p_{i-1}(\cdot)] &= (1 - p_{i-1}(v_i))\frac{p_{i-1}(v_k)}{1 - p_{i-1}(v_i)}  = p_{i-1}(v_k),
            \end{align*}
            as desired.
    
            \item \textbf{Case 2:} $\frac{p_{i-1}(v_k)}{1 - p_{i-1}(v_i)} > \hat p$.
            We have
            \begin{align*}
                \E[p_i(v_k) \mid p_{i-1}(\cdot)] &= (1 - p_{i-1}(v_i))\hat p + p_{i-1}(v_i)\cdot\hat p\cdot\left(\frac{p_{i-1}(v_k)/\hat p - (1 - p_{i-1}(v_i))}{p_{i-1}(v_i)}\right) \\
                &= p_{i-1}(v_k),
            \end{align*}
            as desired.
        \end{itemize}
        The above covers all cases and completes the proof.
    \end{claimproof}
    
    Noting that $p_0(v_k) = \alpha$, applying Claim~\ref{claim: pi k3} recursively completes the proof.
\end{proof}

To assist with the remainder of the proof, we define the following random variables for $0 \leq i < j, k \leq n$.
\begin{align*}
    Q_i(v_j, v_k) &= \mathbf{1}\{v_jv_k \in E(G)\}p_i(v_j)p_i(v_k), \\
    H_i(v_k) &= -p_i(v_k)\log p_i(v_k).
\end{align*}
We refer to $Q_i(v_j, v_k)$ as the \textit{energy} of the pair $(v_j, v_k)$, and $H_i(v_k)$ as the \textit{entropy} of $v_k$.
We begin by computing the expected values of each of the above random variables.

\begin{lemma}\label{lemma: expectations k3}
    For $0 \leq i < j, k \leq n$
    \begin{align*}
        \E[Q_i(v_j, v_k)] &= Q_0(v_j, v_k) = \mathbf{1}\{v_jv_k \in E(G)\}\alpha^2, \\
        \E[H_i(v_k)] &\geq H_0(v_k) - (1+ \eps/10)\sum_{j = 1}^iQ_{j - 1}(v_j, v_k) \geq \alpha\log(1/\alpha) - (1+ \eps/10)d\alpha^2.
    \end{align*}
\end{lemma}

\begin{proof}\stepcounter{ForClaims} \renewcommand{\theForClaims}{\ref{lemma: expectations k3}}
    Let us first consider to $Q_i(v_j, v_k)$.
    Without loss of generality, let us assume that $j < k$ and $v_jv_k \in E(G)$.
    Note the following:
    \[\E[Q_i(v_j, v_k)] = \E[p_i(v_j)p_i(v_k)].\]
    Similar to the proof of Lemma~\ref{lemma: K3 martingale}, it is enough to show that $p_i(v_j)p_i(v_k)$ is a martingale in $i$ for $i = 0, \ldots, j-1$.

    \begin{claim}\label{claim: hi k3}
        $\E[p_i(v_j)p_i(v_k) \mid p_{i-1}(\cdot)] = p_{i-1}(v_j)p_{i-1}(v_k)$ for $i = 1, \ldots, j-1$.
    \end{claim}

    \begin{claimproof}
        Let $M_{i, j}(c) = \frac{p_i(v_j)}{p_{i-1}(v_j)}$ and $M_{i, k}(c) = \frac{p_i(v_k)}{p_{i-1}(v_k)}$ be random variables.
        We have
        \begin{align*}
            \E[p_i(v_j)p_i(v_k) \mid p_{i-1}(\cdot)] &= p_{i-1}(v_j)p_{i-1}(v_k)\E[M_{i, j}(c)M_{i, k}(c)\mid p_{i-1}(\cdot)].
        \end{align*}
        If $v_i \notin N_L(v_j) \cup N_L(v_k)$, we trivially have $\E[M_{i, j}(c)M_{i, k}(c)\mid p_{i-1}(\cdot)] = 1$.
        Additionally, if $v_i \in N_L(v_j) \triangle N_L(v_k)$, an identical argument as in Claim~\ref{claim: pi k3} yields $\E[M_{i, j}(c)M_{i, k}(c)\mid p_{i-1}(\cdot)] = 1$.
        Since $G$ is $K_3$-free, we have $N_L(v_j) \cap N_L(v_k) = \emptyset$ and so the above cover all cases.
    \end{claimproof}

    Noting that $Q_0(v_j,v_k) = \mathbf{1}\{v_jv_k \in E(G)\}\alpha^2$, applying Claim~\ref{claim: hi k3} recursively completes the proof for $Q_i(v_j,v_k)$.

    Next, let us turn our attention to $H_i(v_k)$.
    The following claim will be key to the argument:

    \begin{claim}\label{claim: qi k3}
        For $i = 1, \ldots, k-1$, we have
        \[\E[p_i(v_k)\log p_i(v_k) \mid p_{i-1}(\cdot)] \leq p_{i-1}(v_k)\log p_{i-1}(v_k) + (1+ \eps/10)\mathbf1\{v_iv_k \in E(G)\}p_{i-1}(v_i) p_{i-1}(v_k).\]
    \end{claim}
    
    \begin{claimproof}
        The claim is trivial if $v_i \notin N_L(v_k)$ or $v_i \in B_{i-1}$ so we may assume that $v_i \in N_L(v_k)$ and $v_i \notin B_{i-1}$.
        By convention, we let $0\log 0 = 0$ (indeed, $\lim_{x \to 0}x\log x = 0$).
        As in the proof of Claim~\ref{claim: pi k3}, we split into two cases.
        \begin{itemize}[itemsep=5pt]
            \item \textbf{Case 1:} $\frac{p_{i-1}(v_k)}{1 - p_{i-1}(v_i)} \leq \hat p$.
            We have
            \begin{align*}
                \E[p_i(v_k) \log p_i(v_k)\mid p_{i-1}(\cdot)] &= (1 - p_{i-1}(v_i))\frac{p_{i-1}(v_k)}{1 - p_{i-1}(v_i)}\log \left(\frac{p_{i-1}(v_k)}{1 - p_{i-1}(v_i)}\right) \\
                &= p_{i-1}(v_k)\log p_{i-1}(v_k) + p_{i-1}(v_k) \log \left(\frac{1}{1 - p_{i-1}(v_i)}\right) \\
                &\leq p_{i-1}(v_k)\log p_{i-1}(v_k) + (1+ \eps/10)p_{i-1}(v_i)p_{i-1}(v_k),
            \end{align*}
            where we use the fact that $\log (1/(1-x)) < (1+ \eps/10)x$ for $x \leq \hat p$.
    
            \item \textbf{Case 2:} $\frac{p_{i-1}(v_k)}{1 - p_{i-1}(v_i)} > \hat p$.
            We have
            \begin{align*}
                \E[p_i(v_k)\log p_i(v_k)\mid p_{i-1}(\cdot)] &= (1 - p_{i-1}(v_i))\cdot \hat p \log \hat p \\
                &\qquad \qquad + p_{i-1}(v_i)\left(\frac{p_{i-1}(v_k)/\hat p - (1 - p_{i-1}(v_i))}{p_{i-1}(v_i)}\right)\cdot \hat p \log \hat p \\
                &= p_{i-1}(v_k)\log \hat p \\
                &< p_{i-1}(v_k)\log \left(\frac{p_{i-1}(v_k)}{1 - p_{i-1}(v_i)}\right) \\
                &\leq p_{i-1}(v_k)\log p_{i-1}(v_k) + (1+ \eps/10)p_{i-1}(v_i)p_{i-1}(v_k).
            \end{align*}
        \end{itemize}
        The above covers all cases and completes the proof.
    \end{claimproof}

    It follows from Claim~\ref{claim: qi k3} that
    \[\E[H_i(v_k) \mid p_{i-1}(\cdot)] \geq H_{i-1}(v_k) - (1+ \eps/10)Q_{i-1}(v_i, v_k).\]
    Noting that $H_0(v_k) = \alpha \log (1/\alpha)$, $Q_0(v_j, v_k) = \mathbf{1}\set{v_jv_k \in E(G)}\alpha^2$, and $|N_L(v_k)| \leq d$, applying the above recursively completes the proof.
\end{proof}

With the above in hand, let us show that $\hat p\Pr[v_k \in B_{k-1}]$ is small.
    We have the following:
    \begin{align*}
        -H_{k-1}(v_k) &= p_{k-1}(v_k) \log p_{k-1}(v_k) \\
        &= (p_{k-1}(v_k) \log p_{k-1}(v_k) - p_{k-1}(v_k) \log \alpha +  p_{k-1}(v_k) \log \alpha) \\
        &= p_{k-1}(v_k)\log \alpha + p_{k-1}(v_k) \log (p_{k-1}(v_k)/\alpha).
    \end{align*}
    Note that if $p_{k-1}(v_k) > 0$, then $p_{k-1}(v_k) \geq \alpha$ since a weight either increases, remains the same, or is set to $0$ during each iteration.
    This implies that $p_{k-1}(v_k)\log (p_{k-1}(v_k)/\alpha) \geq 0$.
    It follows that
    \begin{align*}
        -\E[H_{k-1}(v_k)] &= \E[p_{k-1}(v_k)\log \alpha + p_{k-1}(v_k) \log (p_{k-1}(v_k)/\alpha)] \\
        &= \E[p_{k-1}(v_k)]\log \alpha + \E[p_{k-1}(v_k) \log (p_{k-1}(v_k)/\alpha)] \\
        &\geq \alpha\log \alpha + \hat p \log(\hat p / \alpha)\Pr[v_k \in B_{k-1}],
    \end{align*}
    where we use Lemma~\ref{lemma: K3 martingale}.
    Applying Lemma~\ref{lemma: expectations k3}, we obtain
    \begin{align}\label{eq: B bound k3}
       \hat p \log(\hat p / \alpha)\Pr[v_k \in B_{k-1}] \leq (1+ \eps/10)d\alpha^2.
    \end{align}
    Note the following:
    \[\frac{(1+ \eps/10)d\alpha}{\log (\hat p/\alpha)} = \frac{(1+\eps/10)\log d}{(2 + \eps/2)(1-\eps/10)\log \left(1/\alpha\right)} \leq \frac{1}{2}\]
    for $d$ sufficiently large.
    Plugging this into \eqref{eq: B bound k3} implies $\hat p\Pr[v_k \in B_{k-1}] \leq \alpha/2$.
Combining \eqref{eq: K3 main bound} with Lemma~\ref{lemma: expectations k3} and the above computations, we have
\[\Pr[v_k \in I] = \E[p_{k-1}(v_k)] - \hat p \Pr[v_k \in B_{k-1}] \geq \frac{\alpha}{2},\]
completing the proof of Theorem~\ref{theo: k3}.

\section{Locally $r$-colorable graphs: proofs of Theorems~\ref{theo: r colorable ordinary coloring} and~\ref{theo: local demands r colorable}}\label{section: locally colorable}

In this section, we will prove our results pertaining to locally $r$-colorable graphs.
We begin by stating a key result, which shows that a locally $r$-colorable graph admits a $\phi$-coloring with respect to an appropriate demand function $\phi$.

\begin{proposition}\label{prop: r colorable}
    For all $\eps > 0$ and $r \in \N$ there exists $\delta_0 \in \N$ such that the following holds for $\delta \geq \delta_0$ and $n \in \N$.
    Let $G$ be an $n$-vertex locally $r$-colorable graph and let $\phi$ be a demand function for $G$.
    If there exists an ordering $(v_1, \ldots, v_n)$ of the vertices such that $\phi(v_i) \leq \phi(v_j)$ whenever $i < j$ and
    \[\phi(v) \leq \dfrac{\log d^*(v)}{(8+\eps)\,d^*(v)\,\log(2r)}, \qquad \text{for each } v \in V(G),\]
    where $d^*(v) \coloneqq \max\{d_L(v),\,\delta\}$.
    Then $G$ admits a $\phi$-coloring.
    
    Moreover, there is a $\poly(n)$-time algorithm that samples an independent set $I \subseteq V(G)$ such that
    \[\Pr[v\in I] \geq \phi(v), \qquad \text{for each } v \in V(G).\]
\end{proposition}

Before we proceed, we note that the above result implies Theorems~\ref{theo: r colorable ordinary coloring} and~\ref{theo: local demands r colorable}.
Indeed, applying Proposition~\ref{prop: r colorable} with $\phi(v) = \frac{\log d}{(8+\eps)\,d\,\log(2r)}$ for each $v$, $\delta = d$, and a degeneracy ordering of $G$ yields Theorem~\ref{theo: r colorable ordinary coloring}.
Similarly, applying Proposition~\ref{prop: r colorable} ordering the vertices in increasing order of demands, $\delta = \delta(G)$, and noting that $d^*(v) \leq \deg(v)$ yields Theorem~\ref{theo: local demands r colorable}.
The remainder of this section is devoted to the proof of Proposition~\ref{prop: r colorable}.

Let us describe our independent set procedure (see Algorithm~\ref{algorithm: fcp}).
The algorithm takes as input an $n$-vertex locally $r$-colorable graph $G$, a demand function $\phi$ for $G$, and an ordering $(v_1, \ldots, v_n)$ of the vertices $V(G)$, and outputs an independent set in $G$.
We refer the reader to \S\ref{subsection: overview} for an informal overview of the procedure.
Let us now provide a formal description of the algorithm.

\vspace{10pt}
\begin{breakablealgorithm}
\caption{Independent Set Procedure for Locally $r$-Colorable Graphs}\label{algorithm: fcp}
\begin{flushleft}
\textbf{Input}: An $n$-vertex locally $r$-colorable graph $G$, a demand function $\phi$ for $G$, and an ordering $(v_1, \ldots, v_n)$ of the vertices $V(G)$. \\
\textbf{Output}: An independent set in $G$.
\end{flushleft}

\begin{enumerate}[itemsep=5pt]
    \item \textbf{Initialize:} Set $I = B_0 = \emptyset$.
    For each $v \in V(G)$, set $p_0(v) = \alpha_v$, where $\alpha_v \coloneqq 4\phi(v)$ and fix a proper $r$-coloring  $\psi_v\,:\, N(v) \to [r]$ for $G[N(v)]$.

    \item \textbf{Iterate:} For $i = 1, \ldots, n$, do the following:
        \begin{enumerate}[itemsep=5pt]
                \item For each $v_k$ such that $v_k \notin N_R(v_i)$, set $p_i(v_k) = p_{i-1}(v_k)$.
                \item \label{step: ai etai Ji} Let $a_{i} \sim \mathrm{BER}(p_{i-1}(v_i))$, $s_{i} \sim \mathrm{BER}(1/2)$, and let $J_{i}$ be a color class chosen uniformly at random from the coloring induced by $\psi_{v_i}$ on $N_R(v_i)$.
                \item If $a_{i} = 1$ and $s_{i} = 1$, then add $v$ to $I$.
                \item For each $v_k \in N_R(v_i)$, do the following:
                \begin{enumerate}[label=\ep{\roman*}, itemsep=5pt]
                    \item\label{step: not activated} If $a_{i} = 0$ or $p_{i-1}(v_k) = 1$, set $p_i(v_k) = p_{i-1}(v_k)$.

                    \item\label{step: well behaved r} If $2rp_{i-1}(v_k) \leq 1$, make the following update: 
                    \begin{itemize}[itemsep=5pt]
                        \item If $a_{i} = 1$ and $s_{i} = 0$, then set $p_i(v_k) = 2rp_{i-1}(v_k)$ if $v_k\in J_{i}$ and $0$ otherwise.

                        \item If $a_{i} = 1$ and $s_{i} = 1$, then set $p_i(v_k) = 0$.

                    \end{itemize}

                    \item\label{step: ub violated r} If $2rp_{i-1}(v_k) > 1$, make the following update:

                    \begin{itemize}[itemsep=5pt]
                        \item If $a_{i} = 1$ and $s_{i} = 0$, then set $p_i(v_k) = 1$ if $v_k\in J_{i}$.
                        If $v_k \notin J_{i}$, set $p_i(v_k) = 1$ with probability $\mu_{i, k}\coloneqq \frac{p_{i-1}(v_k) - \frac{1}{2r}}{1 - \frac{1}{2r}}$ and $0$ otherwise.
                        
                        \item If $a_{i} = 1$ and $s_{i} = 1$, then set $p_i(v_k) = 1$ with probability $\mu_{i, k}$ and $0$ otherwise.
                    \end{itemize}
                    
                \end{enumerate}
            \item\label{step: bad update} Set $B_i = \set{v_k\,:\, p_i(v_k) = 1,\, k > i}$.
        \end{enumerate}
\end{enumerate}
\end{breakablealgorithm}
\vspace{10pt}

A few remarks are in order.
First, note that Algorithm~\ref{algorithm: fcp} can be implemented in $O(n^2)$ time satisfying the efficiency requirements and performance guarantees outlined in Proposition~\ref{prop: r colorable}.
Additionally, it is easy to see that our update steps ensure $I$ is an independent set.
Furthermore, as a result of Step~\ref{step: not activated}, once a vertex is bad, it remains bad.
Finally, we note that the probabilities $\mu_{i, k}$ are well-defined.
Indeed, $\mu_{i, k} \in [0, 1]$ as long as $p_{i-1}(v_k) \in [1/(2r),\, 1]$ (which is the case when we are in Step~\ref{step: ub violated r}).
It remains to show that $\Pr[v_k \in I]$ is large for each vertex $v_k$.
To this end, note the following:
\begin{align}\label{eq: local r colorable main bound}
    \Pr[v_k \in I] = \E\left[\frac{1}{2}p_{k-1}(v_k)\mathbf{1}\set{v_k \notin B_{k-1}}\right] = \frac{1}{2}\left(\E[p_{k-1}(v_k)] - \hat p \Pr[v_k \in B_{k-1}]\right).
\end{align}

Let us begin by showing that the random variable $p_i(v_k)$ is a martingale in $i$, which in particular implies $\E[p_{k-1}(v_k)] = \alpha_{v_k}$.

\begin{lemma}\label{lemma: local r colorable martingale}
    For $0 \leq i < k \leq n$, we have $\E[p_i(v_k)] = \alpha_{v_k}$.
\end{lemma}

\begin{proof}\stepcounter{ForClaims} \renewcommand{\theForClaims}{\ref{lemma: local r colorable martingale}}
    The following claim shows that $p_i(v_k)$ is a martingale in $i$ for $i = 0, \ldots, k-1$.

    \begin{claim}\label{claim: pi}
        $\E[p_i(v_k) \mid p_{i-1}(\cdot)] = p_{i-1}(v_k)$ for $i = 1, \ldots, k-1$.
    \end{claim}
    
    \begin{claimproof}
        If $v_i \notin N_L(v_k)$ or $\{v_i, v_k\} \cap B_{i-1} \neq \emptyset$, the claim is trivial.
        Suppose $v_i \in N_L(v_k)$ and $\{v_i, v_k\} \cap B_{i-1} = \emptyset$.
        We have two cases to consider.
        \begin{itemize}[itemsep=5pt]
            \item \textbf{Case 1:} $2rp_{i-1}(v_k) \leq 1$.
            We have
            \begin{align*}
                \E[p_i(v_k) \mid p_{i-1}(\cdot)] &= (1 - p_{i-1}(v_i))p_{i-1}(v_k) + p_{i-1}(v_i)\cdot\frac{1}{2}\cdot\frac{1}{r}\,(2rp_{i-1}(v_k)) \\
                &=  p_{i-1}(v_j, c),
            \end{align*}
            as desired.
    
            \item \textbf{Case 2:} $2rp_{i-1}(v_k) > 1$.
            We have
            \begin{align*}
                \E[p_i(v_k) \mid p_{i-1}(\cdot)] &= (1 - p_{i-1}(v_i))p_{i-1}(v_k) + p_{i-1}(v_i)\cdot\frac{1}{2}\cdot\frac{1}{r}\cdot 1 \\
                &\qquad \qquad + p_{i-1}(v_i)\cdot\left(1 - \frac{1}{2r}\right)\mu_{i, k }\cdot 1 \\
                &= p_{i-1}(v_k),
            \end{align*}
            as desired.
            Indeed, we define $\mu_{i, k}$ so that the above equality holds.
        \end{itemize}
        The above covers all cases and completes the proof.
    \end{claimproof}

    Noting that $p_0(v_k) = \alpha_{v_k}$, applying Claim~\ref{claim: pi} recursively completes the proof.
\end{proof}

To assist with the remainder of the proof, we define the following random variables for $0 \leq i < j, k \leq n$.
\begin{align*}
    Q_i(v_j, v_k) &= \mathbf{1}\{v_jv_k \in E(G),\, v_j \notin B_i\}p_i(v_j)p_i(v_k), \\
    H_i(v_k) &= -p_i(v_k)\log p_i(v_k).
\end{align*}
We begin by computing the expected values of each of the above random variables.

\begin{lemma}\label{Lemma: expectations}
    The following hold:
    \begin{align*}
        \E[Q_i(v_j, v_k)] &\leq \mathbf{1}\{v_jv_k \in E(G)\}\alpha_{v_k}^2, & &\text{for } 0 \leq i < j < k \leq n; \\
        \E[H_i(v_k)] &\geq \alpha_{v_k}\log(1/\alpha_{v_k}) - d^*(v_k)\log(2r)\alpha_{v_k}^2,  & &\text{for } 0 \leq i < k \leq n.
    \end{align*}
\end{lemma}

\begin{proof}\stepcounter{ForClaims} \renewcommand{\theForClaims}{\ref{Lemma: expectations}}

    Let us first consider the energy.
    Without loss of generality, let us assume that $v_jv_k \in E(G)$.
    Note the following:
    \[\E[Q_i(v_j, v_k)] = \E[\mathbf{1}\set{v_j \notin B_i}p_i(v_j, c)p_i(v_k)].\]
    Similar to the proof of Lemma~\ref{lemma: local r colorable martingale}, it is enough to show that $\mathbf{1}\set{v_j \notin B_i}p_i(v_j, c)p_i(v_k)$ is a supermartingale in $i$ for $i = 0, \ldots, j-1$.

    \begin{claim}\label{claim: hi}
        $\E[\mathbf{1}\set{v_j \notin B_i}p_i(v_j, c)p_i(v_k) \mid p_{i-1}(\cdot)] \leq \mathbf{1}\set{v_j \notin B_{i-1}}p_{i-1}(v_j, c)p_{i-1}(v_k)$ for $i = 1, \ldots, j-1$.
    \end{claim}

    \begin{claimproof}
        Let $M_{i, j}(c) = \frac{p_i(v_j, c)}{p_{i-1}(v_j, c)}$ and $M_{i, k}(c) = \frac{p_i(v_k)}{p_{i-1}(v_k)}$ be random variables.
        We have
        \[\E[\mathbf{1}\set{v_j \notin B_i}p_i(v_j, c)p_i(v_k) \mid p_{i-1}(\cdot)] = p_{i-1}(v_j, c)p_{i-1}(v_k)\E[\mathbf{1}\set{v_j \notin B_i}M_{i, j}(c)M_{i, k}(c)\mid p_{i-1}(\cdot)].\]
        If $v_i \notin N_L(v_j) \cup N_L(v_k)$, we trivially have 
        \[\E[\mathbf{1}\set{v_j \notin B_i}M_{i, j}(c)M_{i, k}(c)\mid p_{i-1}(\cdot)] = \mathbf{1}\set{v_j \notin B_{i-1}}.\]
        Additionally, if $v_i \in N_L(v_j) \triangle N_L(v_k)$, an identical argument as in Claim~\ref{claim: pi} yields yields the same bound.
        Furthermore, an identical argument applies if $\set{v_j, v_k} \cap B_{i-1} \neq \emptyset$.
        Consider the case that $v_i \in N_L(v_j) \cap N_L(v_k)$ and $\set{v_j, v_k} \cap B_{i-1} = \emptyset$.
        Note that since $v_jv_k \in E(G)$, we cannot have $\{v_j, v_k\} \subseteq J_{i}$.
        With this in hand, we have three cases to consider.
        \begin{itemize}[itemsep=5pt]
            \item \textbf{Case 1:} $2rp_{i-1}(v_j, c), 2rp_{i-1}(v_k) \leq 1$.
            It follows that $\mathbf{1}\set{v_j \notin B_i}M_{i, j}(c)M_{i, k}(c) \neq 0$ if and only if $a_{i} = 0$.
            In particular, we have 
            \begin{align*}
                \E[\mathbf{1}\set{v_j \notin B_i}M_{i, j}(c)M_{i, k}(c) \mid p_{i-1}(\cdot)] = (1-p_{i-1}(v_i)) \leq 1,
            \end{align*}
            as desired.
            
            \item \textbf{Case 2:} $2rp_{i-1}(v_j, c) \leq 1 < 2rp_{i-1}(v_k)$.
            It follows that $\mathbf{1}\set{v_j \notin B_i}M_{i, j}(c)M_{i, k}(c) \neq 0$ if and only if $a_{i} = 0$ or $a_{i} = 1$, $s_{i} = 1$, $v_j \in J_{i}$, and $v_k \in B_i$.
            We have
            \begin{align*}
                &~\E[\mathbf{1}\set{v_j \notin B_i}M_{i, j}(c)M_{i, k}(c) \mid p_{i-1}(\cdot)] \\
                &\qquad = (1-p_{i-1}(v_i)) + p_{i-1}(v_i)\cdot\frac{1}{2}\cdot\frac{1}{r}\cdot \mu_{i, k }\cdot 2r\cdot\frac{1}{p_{i-1}(v_k)} \\
                &\qquad = (1-p_{i-1}(v_i)) + p_{i-1}(v_i)\left(\frac{1 - \frac{1}{2rp_{i-1}(v_k)}}{1 - \frac{1}{2r}}\right) \leq 1,
            \end{align*}
            where the last step follows since $p_{i-1}(v_k) \leq 1$.

            \item \textbf{Case 3:} $2rp_{i-1}(v_j, c) > 1$.
            It follows that $\mathbf{1}\set{v_j \notin B_i}M_{i, j}(c)M_{i, k}(c) \neq 0$ if and only if $a_{i} = 0$.
            In particular, we have 
            \begin{align*}
                \E[\mathbf{1}\set{v_j \notin B_i}M_{i, j}(c)M_{i, k}(c) \mid p_{i-1}(\cdot)] = (1-p_{i-1}(v_i)) \leq 1,
            \end{align*}
            as desired.
        \end{itemize} 
        The above covers all cases and completes the proof.
    \end{claimproof}

    Noting that $Q_0(v_j,v_k) = \mathbf{1}\{v_jv_k \in E(G)\}\alpha_{v_j}\alpha_{v_k}$ and that $\alpha_{v_j} \leq \alpha_{v_k}$ as $\phi(v_j) \leq \phi(v_k)$, applying Claim~\ref{claim: hi} recursively completes the proof for $\E[Q_i(v_j,v_k)]$.

    Next, let us turn our attention to $H_i(v_k)$.
    Note the following:
    \[\E[H_i(v_k)] = -\E[p_i(v_k)\log p_i(v_k)].\]
    The following claim will be key to the argument:

    \begin{claim}\label{claim: qi}
        For $i = 1, \ldots, k-1$, we have
        \begin{align*}
            \E[p_i(v_k)\log p_i(v_k) \mid p_{i-1}(\cdot)] &\leq p_{i-1}(v_k)\log p_{i-1}(v_k) \\
            &\qquad  + \log (2r) \mathbf1\set{v_iv_k \in E(G),\, v_i \notin B_{i-1}}p_{i-1}(v_i) p_{i-1}(v_k).
        \end{align*}
    \end{claim}
    
    \begin{claimproof}
        The claim is trivial if $v_i \notin N_L(v_k)$ or $\set{v_i, v_k} \cap B_{i-1} \neq \emptyset$, so we may assume that $v_i \in N_L(v_k)$ and $\set{v_i, v_k} \cap B_{i-1} = \emptyset$.
        As in the proof of Claim~\ref{claim: pi}, we split into two cases.
        \begin{itemize}[itemsep=5pt]
            \item \textbf{Case 1:} $2rp_{i-1}(v_k) \leq 1$.
            We have
            \begin{align*}
                \E[p_i(v_k) \log p_i(v_k)\mid p_{i-1}(\cdot)] &= (1 - p_{i-1}(v_i))p_{i-1}(v_k)\log p_{i-1}(v_k) \\
                &\qquad \qquad + p_{i-1}(v_i)\cdot\frac{1}{2}\cdot\frac{1}{r}\cdot(2rp_{i-1}(v_k) \log (2rp_{i-1}(v_k))) \\
                &= p_{i-1}(v_k)\log p_{i-1}(v_k) + \log(2r)p_{i-1}(v_i) p_{i-1}(v_k).
            \end{align*}
    
            \item \textbf{Case 2:} $2rp_{i-1}(v_k) > 1$.
            We have
            \begin{align*}
                \E[p_i(v_k)\log p_i(v_k)\mid p_{i-1}(\cdot)] &= (1 - p_{i-1}(v_i))p_{i-1}(v_k)\log p_{i-1}(v_k) \\
                &= p_{i-1}(v_k)\log p_{i-1}(v_k)  + p_{i-1}(v_i)p_{i-1}(v_k) \log \left(\frac{1}{p_{i-1}(v_k)}\right) \\
                &< p_{i-1}(v_k)\log p_{i-1}(v_k) + \log(2r)p_{i-1}(v_i) p_{i-1}(v_k).
            \end{align*}
        \end{itemize}
        The above covers all cases and completes the proof.
    \end{claimproof}

    It follows from Claim~\ref{claim: qi} that
    \[\E[H_i(v_k) \mid p_{i-1}(\cdot)] \geq H_{i-1}(v_k) - \log (2r) Q_{i-1}(v_i, v_k).\]
    Noting that $H_0(v_k) = \alpha_{v_k} \log (1/\alpha_{v_k})$, $Q_0(v_j, v_k) \leq \mathbf{1}\set{v_jv_k \in E(G)}\alpha_{v_k}^2$, and $|N_L(v_k)| \leq d^*(v_k)$, applying the above recursively in conjunction with the expectation bound on $Q_{i-1}(v_i, v_k)$ proved earlier completes the proof.
\end{proof}

With the above in hand, let us show that $\Pr[v_k \in B_{k-1}]$ is small.
    We have the following:
    \begin{align*}
        -H_{k-1}(v_k) &= p_{k-1}(v_k) \log p_{k-1}(v_k) \\
        &= \left(p_{k-1}(v_k) \log p_{k-1}(v_k) - p_{k-1}(v_k) \log \alpha_{v_k} +  p_{k-1}(v_k) \log \alpha_{v_k}\right) \\
        &= p_{k-1}(v_k)\log \alpha_{v_k} + p_{k-1}(v_k) \log (p_{k-1}(v_k)/\alpha_{v_k})
    \end{align*}
    Note that if $p_{k-1}(v_k) > 0$, then $p_{k-1}(v_k) \geq \alpha_{v_k}$ since a weight either increases, remains the same, or is set to $0$ during each iteration.
    This implies that all $p_{k-1}(v_k)\log (p_{k-1}(v_k)/\alpha_{v_k}) \geq 0$.
    It follows that
    \begin{align*}
        -\E[H_{k-1}(v_k)] &= \E[p_{k-1}(v_k)\log \alpha_{v_k} + p_{k-1}(v_k) \log (p_{k-1}(v_k)/\alpha_{v_k})] \\
        &= \E[p_{k-1}(v_k)]\log \alpha_{v_k} + \E[p_{k-1}(v_k) \log (p_{k-1}(v_k)/\alpha_{v_k})] \\
        &\geq \alpha_{v_k}\log \alpha_{v_k} + \log(1/ \alpha_{v_k})\Pr[v_k \in B_{k-1}],
    \end{align*}
    where we use Lemma~\ref{lemma: local r colorable martingale}.
    Applying Lemma~\ref{Lemma: expectations}, we obtain
    \begin{align}\label{eq: B bound local r colorable}
       \log(1 / \alpha_{v_k})\Pr[v_k \in B_{k-1}] \leq d^*(v_k)\log(2r)\alpha_{v_k}^2.
    \end{align}
    We next argue that
    \begin{align}\label{eq: r colorable}
        \frac{d^*(v_k)\log(2r)\alpha_{v_k}}{\log (1 / \alpha_{v_k})} \leq \frac{1}{2}.
    \end{align}
    Noting that the function $x/\log (1/x)$ is increasing on $[0, 1]$ and
    \[\alpha_{v_k} = 4\phi(v_k) \leq \frac{1}{(2 + \eps/4)}\frac{\log d^*(v_k)}{d^*(v_k)\log(2r)},\]
    the bound in \eqref{eq: r colorable} follows for $\delta$ sufficiently large by replacing $\alpha_{v_k}$ with the expression above, yielding the desired bound.
    In particular, we have $\Pr[v_k \in B_{k-1}] \leq \alpha_{v_k}/2$.

Combining \eqref{eq: r colorable} with Lemma~\ref{Lemma: expectations} and the above computations, we have
\[\Pr[v_k \in I] = \frac{1}{2}\left(\E[p_{k-1}(v_k)] - \Pr[v_k \in B_{k-1}]\right) \geq \frac{\alpha_{v_k}}{4} = \phi(v),\]
completing the proof of Proposition~\ref{prop: r colorable}.

\section{High-girth hypergraphs: proofs of Theorems~\ref{theo: hypergraph free ordinary coloring} and~\ref{theo: hypergraph local demands}}\label{section: hypergraph}

In this section, we will prove our results pertaining to hypergraphs with girth at least $4$.
We begin by stating a key result, which shows that such a hypergraph admits a $\phi$-coloring with respect to an appropriate demand function $\phi$.

\begin{proposition}\label{prop: hypergraph}
    For all $\eps > 0$ and integer $r \geq 2$ there exists $\delta_0 \in \N$ such that the following holds for $\delta \geq \delta_0$ and $n \in \N$.
    Let $H$ be an $n$-vertex $r$-uniform hypergraph with girth at least $4$ and let $\phi$ be a demand function for $H$.
    If there exists an ordering $(v_1, \ldots, v_n)$ of the vertices such that $\phi(v_i) \leq \phi(v_j)$ whenever $i < j$ and
    \[\phi(v) \leq (1-\eps)\left(\frac{r-1}{r}\right)\left(\dfrac{\log d^*(v)}{r(r-1)^2d^*(v)}\right)^{\frac{1}{r-1}}, \qquad \text{for each } v \in V(H),\]
    where $d^*(v) \coloneqq \max\{d_L(v)/(r-1),\,\delta\}$.
    Then $H$ admits a $\phi$-coloring.
    
    Moreover, there is a $\poly(n)$-time algorithm that samples an independent set $I \subseteq V(H)$ such that
    \[\Pr[v\in I] \geq \phi(v), \qquad \text{for each } v \in V(H).\]
\end{proposition}

Before we proceed, we note that the above result implies Theorems~\ref{theo: hypergraph free ordinary coloring} and~\ref{theo: hypergraph local demands}.
Indeed, applying Proposition~\ref{prop: hypergraph} with $\phi(v) = (1-\eps)\left(1 - \frac{1}{r}\right)\left(\frac{\log d}{r(r-1)^2d}\right)^{\frac{1}{r-1}}$ for each $v$, $\delta = d$, and a degeneracy ordering of $H$ yields Theorem~\ref{theo: hypergraph free ordinary coloring}.
Similarly, applying Proposition~\ref{prop: hypergraph} ordering the vertices in increasing order of demands, $\delta = \delta(G)$, and noting that $d^*(v) \leq \deg(v)$ yields Theorem~\ref{theo: hypergraph local demands}.
The remainder of this section is devoted to the proof of Proposition~\ref{prop: hypergraph}.

Let us describe our independent set procedure (see Algorithm~\ref{algorithm: fcp hypergraph}).
The algorithm takes as input an $n$-vertex $r$-uniform hypergraph $H$ with girth at least $4$, a demand function $\phi$ for $H$, an ordering $(v_1, \ldots, v_n)$ of the vertices $V(H)$, and a parameter $\eps > 0$, and outputs an independent set $I$ in $H$.
Before we state the algorithm rigorously, we provide an informal overview.
We initially assign each vertex $v$ a weight $p_0(v) = \alpha_v$.
Iterating through the vertices in a degeneracy ordering of $H$, we add $v_i$ to $I$ with probability $p_{i-1}(v_i)$.
We update the weights $p_i(v_k)$ for each $k > i$ according to the outcome of the event $\set{v_i \in I}$.
More formally, let $v_k$ be a right-neighbor of $v_i$, and let $e$ be the unique edge containing $\set{v_i, v_k}$.
Additionally, let $f \subseteq e$ consist of the vertices of $e$ that appear before $v_i$ in the degeneracy ordering.
We update $p_i(v_k)$ as follows:

\begin{itemize}[itemsep=0.2cm]
    \item if $f \not\subseteq I$, then $p_i(v_k) = p_{i-1}(v_k)$;
    \item if $f \subseteq I$ and $v_i \notin I$, then $p_i(v_k) = \dfrac{p_{i-1}(v_k)}{1 - \prod_{u \in e \setminus (f \cup \{v_k\})}p_{i-1}(u)}$;
    \item if $f \cup \{v_i\} \subseteq I$, then $p_i(v_k) = p_{i-1}(v_k)\left(\dfrac{1 - \prod_{u \in e \setminus (f \cup \{v_i, v_k\})}p_{i-1}(u)}{1 - \prod_{u \in e \setminus (f \cup \{v_k\})}p_{i-1}(u)}\right)$.
\end{itemize}
We refer the reader to \S\ref{subsection: overview} for a motivation to the above update rule.
Finally, we ensure $p_i(v_k) \leq \alpha_{v_k}^{\gamma}$ and $p_i(v_k) \geq \alpha_{v_k}^{1+\gamma}$ for $\gamma = \eps/(1000r)$, and define equalizing coin flips accordingly.

Let us now provide a formal description of the algorithm.

\vspace{10pt}
\begin{breakablealgorithm}
\caption{Independent Set Procedure for High Girth Hypergraphs}\label{algorithm: fcp hypergraph}
\begin{flushleft}
\textbf{Input}: An $n$-vertex $r$-uniform hypergraph $H$ with girth at least $4$, a demand function $\phi$ for $H$, an ordering $(v_1, \ldots, v_n)$ of the vertices $V(H)$, and a parameter $\eps > 0$. \\
\textbf{Output}: An independent set $I$ in $H$.
\end{flushleft}

\begin{enumerate}[itemsep=5pt]
    \item \textbf{Initialize:} Set $I = B_0^\mu = B_0^\ell = \emptyset$.
    For each $v \in V(H)$, set $p_0(v) = \alpha_v$, where $\alpha_v \coloneqq (1+\eps/10)\frac{r}{r-1}\phi(v)$.

    \item \textbf{Iterate:} For $i = 1, \ldots, n$, do the following:
        \begin{enumerate}[itemsep=5pt]
                \item For each $v_k$ such that $v_k \notin N_R(v_i)$, set $p_i(v_k) = p_{i-1}(v_k)$.
                \item \label{step: ai} Let $a_{i} \sim \mathrm{BER}(p_{i-1}(v_i))$. If $a_{i} = 1$, then add $v_i$ to $I$.
                \item For each edge $e \in E(H)$ containing $v_i$ as an internal vertex, do the following:
                \begin{enumerate}[label=\ep{\roman*}, itemsep=5pt]
                    \item Let $f \subseteq e$ consist of the vertices of $e$ that appear before $v_i$, let $v_k$ be the right-most vertex in $e$, and let $X_{i, k} \coloneqq \prod_{u \in e \setminus (f \cup \{v_k\})}p_{i-1}(u)$ and $X_{i, k}' \coloneqq X_{i, k}/p_{i-1}(v_i)$.
                    \item\label{step: not activated hypergraph} If $f \not\subseteq I$, or $v_k \in B_{i-1}^\mu \cup B_{i-1}^\ell$, or $p_{i-1}(v_k) = 0$, set $p_i(v_k) = p_{i-1}(v_k)$.

                    \item\label{step: well behaved} If $p_{i-1}(v_k) \leq (1-X_{i, k})\alpha_{v_k}^{\gamma}$ and either $p_{i-1}(v_k)(1-X_{i, k}') \geq (1-X_{i, k})\alpha_{v_k}^{1+\gamma}$ or $X_{i, k}' = 1$, make the following update:

                    \begin{itemize}[itemsep=5pt]
                        \item If $a_{i} = 1$ and $f \subseteq I$, set $p_i(v_k) = p_{i-1}(v_k)\left(\frac{1 - X_{i, k}'}{1 - X_{i, k}}\right)$.
                        \item If $a_{i} = 0$ and $f \subseteq I$, set $p_i(v_k) = \frac{p_{i-1}(v_k)}{1 - X_{i, k}}$.
                    \end{itemize}

                    \item\label{step: ub violated} If $p_{i-1}(v_k) > (1-X_{i, k})\alpha_{v_k}^{\gamma}$, make the following update:

                    \begin{itemize}[itemsep=5pt]
                        \item If $a_{i} = 1$ and $f \subseteq I$, set $p_i(v_k) = \alpha_{v_k}^{\gamma}$ with probability $\mu_{i, k}\coloneqq\left(\frac{1-p_{i-1}(v_i)}{p_{i-1}(v_i)}\right)\left(\frac{p_{i-1}(v_k)- (1 - X_{i, k})\alpha_{v_k}^{\gamma}}{(1-X_{i, k})\alpha_{v_k}^{\gamma} - (1-X_{i, k}')p_{i-1}(v_k)}\right)$ and $p_{i-1}(v_k)\left(\frac{1 - X_{i, k}'}{1 - X_{i, k}}\right)$ otherwise.
                        \item If $a_{i} = 0$ and $f \subseteq I$, set $p_i(v_k) = \alpha_{v_k}^{\gamma}$.
                    \end{itemize}

                    \item\label{step: lb violated} If $0 < p_{i-1}(v_k)(1-X_{i, k}') < (1-X_{i, k})\alpha_{v_k}^{1+\gamma}$, make the following update:

                    \begin{itemize}[itemsep=5pt]
                        \item If $a_{i} = 1$ and $f \subseteq I$, set $p_i(v_k) = \alpha_{v_k}^{1+\gamma}$.
                        \item If $a_{i} = 0$ and $f \subseteq I$, set $p_i(v_k) = \frac{p_{i-1}(v_k)}{1 - X_{i, k}}$ with probability $\ell_{i, k}\coloneqq \left(\frac{1 - X_{i, k}}{1 - p_{i-1}(v_i)}\right)\left(1 - \alpha_{v_k}^{1+\gamma}\frac{p_{i-1}(v_i)}{p_{i-1}(v_k)}\right)$ and $0$ otherwise.
                    \end{itemize}

                \end{enumerate}
            \item\label{step: bad update hypergraph} 
            Set $B_i^\mu = \set{v_k\,:\, p_i(v_k) = \alpha_{v_k}^{\gamma},\, k > i}$ and $B_i^\ell = \set{v_k\,:\, p_i(v_k) = \alpha_{v_k}^{1+\gamma},\, k > i}$.
            
        \end{enumerate}
\end{enumerate}
\end{breakablealgorithm}
\vspace{10pt}

A few remarks are in order.
First, note that Algorithm~\ref{algorithm: fcp hypergraph} can be implemented in $O(rn^2)$ time satisfying the efficiency requirements and performance guarantees outlined in Proposition~\ref{prop: hypergraph}.
Additionally, given an edge $e = \set{v_{i_1}, \ldots, v_{i_r}}$ such that $i_1 < i_2 < \cdots < i_r$, we have $X_{i_{r-1}, i_r}' = 1$.
In particular, $p_{i_r - 1}(v_{i_r}) = 0$ if $\set{v_{i_l}\,:\, l < r} \subseteq I$.
This ensures $I$ is an independent set in $H$.
Furthermore, we note that the probabilities $\mu_{i, k}$ and $\ell_{i, k}$ are well-defined.
Indeed, $\mu_{i, k} \leq 1$ always, and $\mu_{i, k} > 0$ if and only if $p_{i-1}(v_k) > (1-X_{i, k})\alpha_{v_k}^{\gamma}$.
Similarly, $\ell_{i, k} \geq 0$ always and $\ell_{i, k} < 1$ if and only if $p_{i-1}(v_k)(1-X_{i, k}') < (1-X_{i, k})\alpha_{v_k}^{1+\gamma}$.
Finally, we note that by our choice of lower and upper bounds on $p_i(\cdot)$, we can never have both
\[0 < p_{i-1}(v_k)(1-X_{i, k}') < (1-X_{i, k})\alpha_{v_k}^{1+\gamma}, \qquad  \text{and} \qquad p_{i-1}(v_k) > (1-X_{i, k})\alpha_{v_k}^{\gamma},\]
for $\delta$ sufficiently large.
In particular, Algorithm~\ref{algorithm: fcp hypergraph} covers all cases.
It remains to show that $\Pr[v_k \in I]$ is large for each vertex $v_k$.
To this end, note the following:
\begin{align}\label{eq: hypergraph main bound}
    \Pr[v_k \in I] = \E[p_{k-1}(v_k)\mathbf{1}\set{v_k \notin B_{k-1}^\mu \cup B_{k-1}^{\ell}}] = \E[p_{k-1}(v_k)] - \alpha_{v_k}^{\gamma} \Pr[v_k \in B_{k-1}^\mu] - \alpha_{v_k}^{1 + \gamma}\Pr[v_k \in B_{k-1}^{\ell}].
\end{align}

To assist with the proofs, we let $\Omega_i$ denote the outcomes of the random choices at $v_i$, i.e., the activations and (possible) equalizing coin flips that occur when processing $v_i$.
Let us begin by showing that the random variable $p_i(v_k)$ is a martingale in $i$, which in particular implies $\E[p_{k-1}(v_k)] = \alpha_{v_k}$.

\begin{lemma}\label{lemma: expectation potential}
    For $0 \leq i < k \leq n$, we have $\E[p_i(v_k)] = \alpha_{v_k}$.
\end{lemma}

\begin{proof}\stepcounter{ForClaims} \renewcommand{\theForClaims}{\ref{lemma: expectation potential}}
    The following claim shows that $p_i(v_k)$ is a martingale in $i$ for $i = 0, \ldots, k-1$.

    \begin{claim}\label{claim: pi hypergraph}
        $\E[p_i(v_k) \mid \Omega_1, \ldots, \Omega_{i-1}] = p_{i-1}(v_k)$ for $i = 1, \ldots, k-1$.
    \end{claim}
    
    \begin{claimproof}
        If $v_i \notin N_L(v_k)$, the claim is trivial.
        Suppose $v_i \in N_L(v_k)$.
        Let $e$ be the unique edge containing $v_i$ and $v_k$ (the edge must be unique since $H$ is linear).
        Let $f \subseteq e$ consist of the vertices of $e$ that appear before $v_i$ in the degeneracy ordering, and let $X_{i, k}$, $X_{i, k}'$, $\mu_{i, k}$, and $\ell_{i, k}$ be defined as in Algorithm~\ref{algorithm: fcp hypergraph}.
        If $f \not \subseteq I$, $v_k \in B_{i-1}^\mu \cup B_{i-1}^\ell$, or $p_{i-1}(v_k) = 0$, the claim is trivial and so we may assume neither of these events occur.
        We have three cases to consider.
        \begin{itemize}[itemsep=5pt]
            \item \textbf{Case 1:} $p_{i-1}(v_k) \leq (1-X_{i, k})\alpha_{v_k}^{\gamma}$ and either $p_{i-1}(v_k)(1-X_{i, k}') \geq (1-X_{i, k})\alpha_{v_k}^{1+\gamma}$ or $X_{i, k}' = 1$.
            We have
            \begin{align*}
                \E[p_i(v_k) \mid \Omega_1, \ldots, \Omega_{i-1}] &= (1 - p_{i-1}(v_i))\frac{p_{i-1}(v_k)}{1 - X_{i, k}} + p_{i-1}(v_i)p_{i-1}(v_k)\left(\frac{1 - X_{i, k}'}{1 - X_{i, k}}\right) \\
                &= p_{i-1}(v_k),
            \end{align*}
            as desired.
    
            \item \textbf{Case 2:} $p_{i-1}(v_k) > (1-X_{i, k})\alpha_{v_k}^{\gamma}$.
            We have
            \begin{align*}
                &~\E[p_i(v_k) \mid \Omega_1, \ldots, \Omega_{i-1}] \\
                &\qquad = (1 - p_{i-1}(v_i))\alpha_{v_k}^{\gamma} + p_{i-1}(v_i)\left(p_{i-1}(v_k)\left(\frac{1 - X_{i, k}'}{1 - X_{i, k}}\right)(1 - \mu_{i, k}) + \mu_{i, k}\alpha_{v_k}^{\gamma}\right) \\
                &\qquad = p_{i-1}(v_k),
            \end{align*}
            as desired.
            Indeed, we define $\mu_{i, k}$ so that the above equality holds.

            \item \textbf{Case 3:} $0 < p_{i-1}(v_k)(1-X_{i, k}') < (1-X_{i, k})\alpha_{v_k}^{1+\gamma}$.
            We have
            \begin{align*}
                \E[p_i(v_k) \mid \Omega_1, \ldots, \Omega_{i-1}]
                &= (1 - p_{i-1}(v_i))\ell_{i, k}\frac{p_{i-1}(v_k)}{1 - X_{i, k}} + p_{i-1}(v_i)\alpha_{v_k}^{1+\gamma} \\
                &= p_{i-1}(v_k),
            \end{align*}
            as desired.
            Indeed, we define $\ell_{i, k}$ so that the above equality holds.
        \end{itemize}
        The above covers all cases and completes the proof.
    \end{claimproof}

    Noting that $p_0(v_k) = \alpha_{v_k}$, applying Claim~\ref{claim: pi hypergraph} recursively completes the proof.
\end{proof}

As in the proof of Proposition~\ref{prop: r colorable}, we will define certain quantities to assist with our proof.
First, the entropy of a vertex is defined identically as in \S\ref{section: K3} and \S\ref{section: locally colorable}:
\begin{align*}
    H_i(v_k) &= -p_i(v_k)\log p_i(v_k) & &\text{for } 0 \leq i < k \leq n.
\end{align*}
The definition of the energy of a pair of vertices is rather more involved.
Let $0 \leq i < j < k \leq n$ and let $e$ be an edge containing $v_j$ and $v_k$ such that $v_k$ is the rightmost vertex of $e$.
Let $f \subseteq e$ be the vertices $v_\ell \in e$ such that $\ell \leq i$.
We define
\[Q_i(v_j, v_k) = \mathbf{1}\set{f \subseteq I\text{ and } \set{v_j, v_k} \cap B_{i}^\mu\cup B_{i}^\ell = \emptyset}\prod_{u \in e \setminus f}p_i(u).\]

We now compute the expected values of each of the above random variables.
We begin with the energy of a pair of vertices.

\begin{lemma}\label{lemma: expectation energy}
    For $0 \leq i < j < k \leq n$, we have $\E[Q_i(v_j, v_k)] \leq \mathbf{1}\{v_j\in N_L(v_k)\}\alpha_{v_k}^r$.
\end{lemma}

\begin{proof}\stepcounter{ForClaims} \renewcommand{\theForClaims}{\ref{lemma: expectation energy}}
    If $v_j \notin N_L(v_k)$, the claim is trivial and so we let $e$ be the unique edge containing $v_j$ and $v_k$ such that $v_k$ is the rightmost vertex of $e$ (the edge is unique since $H$ is linear).
    Let $f \subseteq e$ be the vertices $v_\ell \in e$ such that $\ell \leq i$.
    We have
    \[\E[Q_i(v_j, v_k)] = \E\left[\mathbf{1}\set{f \subseteq I\text{ and } \set{v_j, v_k} \cap B_{i}^\mu\cup B_{i}^\ell = \emptyset}\prod_{u \in e \setminus f}p_i(u)\right].\]
    The following claim will assist with our proof.

    \begin{claim}\label{claim: hi hypergraph}
        \begin{align*}
            &~\E\left[\mathbf{1}\set{f \subseteq I\text{ and } \set{v_j, v_k} \cap B_{i}^\mu\cup B_{i}^\ell = \emptyset}\prod_{u \in e \setminus f}p_i(u)\,\bigg|\, \Omega_1, \ldots, \Omega_{i-1}\right] \\
            & \qquad \leq \mathbf{1}\set{(f\setminus \set{v_i}) \subseteq I \text{ and } \set{v_j, v_k} \cap B_{i-1}^\mu\cup B_{i-1}^\ell = \emptyset}\prod_{u \in e \setminus (f\setminus \set{v_i})}p_{i-1}(u).
        \end{align*}
    \end{claim}

    \begin{claimproof}
        The claim is trivial if $v_i \notin \cup_{u\in e}N_L(u)$ or $\set{v_i, v_j, v_k} \cap B_{i-1}^\mu\cup B_{i-1}^\ell \neq \emptyset$ so we may assume none of these events occur.
        Since $H$ has no $2$- or $3$-cycles, if $v_i\notin e$, then $v_i \in N_L(u)$ for at most one vertex $u \in e$.
        In this case, an identical argument as in the proof of Claim~\ref{claim: pi hypergraph} completes the proof.
        It remains to consider the case that $v_i \in e$ (and hence, in $f$).
        Let $X_{i, k}$, $X_{i, k}'$, $\mu_{i, k}$, and $\ell_{i, k}$ be defined as in Algorithm~\ref{algorithm: fcp hypergraph}.
        If $u \notin I$ for some $u \in f\setminus \set{v_i}$ or $p_{i-1}(v_k) = 0$, the claim is trivial.
        If not, we have three cases to consider.
        \begin{itemize}[itemsep=5pt]
            \item \textbf{Case 1:} $p_{i-1}(v_k) \leq (1-X_{i, k})\alpha_{v_k}^{\gamma}$ and either $p_{i-1}(v_k)(1-X_{i, k}') \geq (1-X_{i, k})\alpha_{v_k}^{1+\gamma}$ or $X_{i, k}' = 1$.
            We have
            \begin{align*}
                &~\E\left[\mathbf{1}\set{f \subseteq I\text{ and } \set{v_j, v_k} \cap B_{i}^\mu\cup B_{i}^\ell = \emptyset}\prod_{u \in e \setminus f}p_i(u)\,\bigg|\, \Omega_1, \ldots, \Omega_{i-1}\right] \\
                &\qquad = \mathbf{1}\set{(f\setminus \set{v_i}) \subseteq I} p_{i-1}(v_i)X_{i, k}p_{i-1}(v_k)\left(\frac{1 - X_{i, k}'}{1 - X_{i, k}}\right) \\
                &\qquad \leq \mathbf{1}\set{\forall v \in f\setminus\set{v_i}}\prod_{u \in e \setminus (f\setminus \set{v_i})}p_{i-1}(u),
            \end{align*}
            where we use the fact that $X_{i, k} \leq X_{i, k}'$.
    
            \item \textbf{Case 2:} $p_{i-1}(v_k) > (1-X_{i, k})\alpha_{v_k}^{\gamma}$.
            We have
            \begin{align*}
                &~\E\left[\mathbf{1}\set{f \subseteq I\text{ and } \set{v_j, v_k} \cap B_{i}^\mu\cup B_{i}^\ell = \emptyset}\prod_{u \in e \setminus f}p_i(u)\,\bigg|\, \Omega_1, \ldots, \Omega_{i-1}\right] \\
                &\qquad = \mathbf{1}\set{(f\setminus \set{v_i}) \subseteq I} p_{i-1}(v_i)X_{i, k}p_{i-1}(v_k)\left(\frac{1 - X_{i, k}'}{1 - X_{i, k}}\right)(1 - \mu_{i, k}) \\
                &\qquad \leq \mathbf{1}\set{(f\setminus \set{v_i}) \subseteq I}\prod_{u \in e \setminus (f\setminus \set{v_i})}p_{i-1}(u),
            \end{align*}
            where we use the fact that $X_{i, k} \leq X_{i, k}'$.

            \item \textbf{Case 3:} $0 < p_{i-1}(v_k)(1-X_{i, k}') < (1-X_{i, k})\alpha_{v_k}^{1+\gamma}$.
            We have
            \begin{align*}
                &~\E\left[\mathbf{1}\set{f \subseteq I\text{ and } \set{v_j, v_k} \cap B_{i}^\mu\cup B_{i}^\ell = \emptyset}\prod_{u \in e \setminus f}p_i(u)\,\bigg|\, \Omega_1, \ldots, \Omega_{i-1}\right]  = 0.
            \end{align*}
        \end{itemize}
        The above covers all cases and completes the proof.
    \end{claimproof}

    Noting that $Q_0(v_j,v_k) = \prod_{u \in e}\alpha_{u} \leq \alpha_{v_k}^r$ as the demands are increasing, applying Claim~\ref{claim: hi hypergraph} recursively completes the proof.
\end{proof}

Next, we consider the entropy at each vertex.

\begin{lemma}\label{Lemma: expectation entropy}
    For $0 \leq i < j < k \leq n$, we have $\E[H_i(v_k)] \geq \alpha_{v_k}\log(1/\alpha_{v_k}) - (1+\gamma)(r-1)d^*(v_k)\alpha_{v_k}^r$.
\end{lemma}

\begin{proof}\stepcounter{ForClaims} \renewcommand{\theForClaims}{\ref{Lemma: expectation entropy}}
    Note the following:
    \[\E[H_i(v_k)] = -\E[p_i(v_k)\log p_i(v_k)].\]
    The following claim will be key to the argument:

    \begin{claim}\label{claim: qi hypergraph}
        For $i = 1, \ldots, k-1$, we have the following:
        \begin{enumerate}
            \item If $v_i \notin N_L(v_k)$ or $\set{v_i, v_k} \cap B_{i-1}^\mu\cup B_{i-1}^\ell \neq \emptyset$, then 
            \[\E[p_i(v_k)\log p_i(v_k) \mid \Omega_1, \ldots, \Omega_{i-1}] = p_{i-1}(v_k)\log p_{i-1}(v_k).\]
            \item Else if there exists an edge $e \ni \set{v_i, v_k}$ such that $v_i \in N_L(v_k)$, then
            \begin{align*}
                &~\E[p_i(v_k)\log p_i(v_k) \mid \Omega_1, \ldots, \Omega_{i-1}] \\
                &\qquad \leq p_{i-1}(v_k)\log p_{i-1}(v_k) + (1+\gamma)\mathbf1\{f \subseteq I\}X_{i, k}p_{i-1}(v_k),
            \end{align*}
            where $X_{i, k}$ is as defined in Algorithm~\ref{algorithm: fcp hypergraph} and $f \subseteq e$ consists of the vertices in $e$ appearing before $v_i$ in the degeneracy ordering.
        \end{enumerate}
            
    \end{claim}
    
    \begin{claimproof}
        The claim is trivial in the former case and so we may assume that $v_i \in N_L(v_k)$ and $\set{v_i, v_k} \cap B_{i-1}^\mu\cup B_{i-1}^\ell = \emptyset$.
        Let $e$ be the unique edge containing $v_i$ and $v_k$, let $f \subseteq e$ be the vertices $v_j \in e$ such that $j < i$, and let $X_{i, k}$, $X_{i, k}'$, $\mu_{i, k}$, and $\ell_{i, k}$ be defined as in Algorithm~\ref{algorithm: fcp hypergraph}.
        If $u \notin I$ for some $u \in f$ or $p_{i-1}(v_k) = 0$, the claim is trivial and so we may assume neither of these events occur.
        As in the proof of Claim~\ref{claim: pi hypergraph}, we split into three cases.
        \begin{itemize}[itemsep=5pt]
            \item \textbf{Case 1:} $p_{i-1}(v_k) \leq (1-X_{i, k})\alpha_{v_k}^{\gamma}$ and either $p_{i-1}(v_k)(1-X_{i, k}') \geq (1-X_{i, k})\alpha_{v_k}^{1+\gamma}$ or $X_{i, k}' = 1$.
            We have
            \begin{align*}
                &~\E[p_i(v_k) \log p_i(v_k)\mid \Omega_1, \ldots, \Omega_{i-1}] \\
                &\qquad = (1 - p_{i-1}(v_i))\frac{p_{i-1}(v_k)}{1 - X_{i, k}}\log \left(\frac{p_{i-1}(v_k)}{1 - X_{i, k}}\right) \\
                &\qquad \qquad +  p_{i-1}(v_i)p_{i-1}(v_k)\left(\frac{1 - X_{i, k}'}{1 - X_{i, k}}\right)\log \left(p_{i-1}(v_k)\left(\frac{1 - X_{i, k}'}{1 - X_{i, k}}\right)\right) \\
                &\qquad = p_{i-1}(v_k)\log p_{i-1}(v_k) + p_{i-1}(v_k) \log \left(\frac{1}{1 - X_{i, k}}\right) \\
                &\qquad \qquad + p_{i-1}(v_i)p_{i-1}(v_k)\left(\frac{1 - X_{i, k}'}{1 - X_{i, k}}\right)\log \left(1 - X_{i, k}'\right) \\
                &\qquad \leq p_{i-1}(v_k)\log p_{i-1}(v_k) + (1+\gamma)X_{i, k}p_{i-1}(v_k),
            \end{align*}
            where we use the fact that $\log (1/(1-x)) < (1+\gamma)x$ for $x$ sufficiently small and $x\log x \leq 0$ for $x \in [0, 1]$.
    
            \item \textbf{Case 2:} $p_{i-1}(v_k) > (1-X_{i, k})\alpha_{v_k}^{\gamma}$.
            We have
            \begin{align*}
                &~\E[p_i(v_k)\log p_i(v_k)\mid p_{i-1}(\cdot)] \\
                &\qquad = (1 - p_{i-1}(v_i))\alpha_{v_k}^{\gamma} \log \left(\alpha_{v_k}^{\gamma}\right) + p_{i-1}(v_i)\mu_{i, k}\alpha_{v_k}^{\gamma}\log \left(\alpha_{v_k}^{\gamma}\right) \\
                &\qquad \qquad \qquad + p_{i-1}(v_i)(1-\mu_{i, k})p_{i-1}(v_k)\left(\frac{1 - X_{i, k}'}{1 - X_{i, k}}\right)\log \left(p_{i-1}(v_k)\left(\frac{1 - X_{i, k}'}{1 - X_{i, k}}\right)\right) \\
                &\qquad \leq (1 - p_{i-1}(v_i))\alpha_{v_k}^{\gamma} \log \left(\alpha_{v_k}^{\gamma}\right) + p_{i-1}(v_i)\mu_{i, k}\alpha_{v_k}^{\gamma} \log \left(\alpha_{v_k}^{\gamma}\right) \\
                &\qquad \qquad \qquad + p_{i-1}(v_i)(1-\mu_{i, k})p_{i-1}(v_k)\left(\frac{1 - X_{i, k}'}{1 - X_{i, k}}\right)\log \left(\alpha_{v_k}^{\gamma}\right),
            \end{align*}
            since $p_{i-1}(v_k)\left(\frac{1 - X_{i, k}'}{1 - X_{i, k}}\right) \leq p_{i-1}(v_k) \leq \alpha_{v_k}^{\gamma}$.
            Noting that
            \[(1 - p_{i-1}(v_i))\alpha_{v_k}^{\gamma}  + p_{i-1}(v_i)\mu_{i, k}\alpha_{v_k}^{\gamma} + p_{i-1}(v_i)(1-\mu_{i, k})p_{i-1}(v_k)\left(\frac{1 - X_{i, k}'}{1 - X_{i, k}}\right) = p_{i-1}(v_k),\]
            the above simplifies to
            \begin{align*}
                \E[p_i(v_k)\log p_i(v_k)\mid p_{i-1}(\cdot)] &\leq p_{i-1}(v_k)\log\left(\alpha_{v_k}^{\gamma}\right) \\
                &\leq p_{i-1}(v_k)\log\left(\frac{p_{i-1}(v_k)}{1 - X_{i, k}}\right) \\
                &\leq p_{i-1}(v_k)\log p_{i-1}(v_k) + (1+\gamma)X_{i, k}p_{i-1}(v_k),
            \end{align*}
            as desired.

            \item \textbf{Case 3:} $0 < p_{i-1}(v_k)(1-X_{i, k}') < (1-X_{i, k})\alpha_{v_k}^{1+\gamma}$.
            We have
            \begin{align*}
                &~\E[p_i(v_k)\log p_i(v_k)\mid p_{i-1}(\cdot)] \\
                &\qquad = (1 - p_{i-1}(v_i))\ell_{i, k}\frac{p_{i-1}(v_k)}{1 - X_{i, k}}\log\left(\frac{p_{i-1}(v_k)}{1 - X_{i, k}}\right) + p_{i-1}(v_i)\alpha_{v_k}^{1+\gamma}\log \left(\alpha_{v_k}^{1+\gamma}\right) \\
                &\qquad < \log\left(\frac{p_{i-1}(v_k)}{1 - X_{i, k}}\right)\left((1 - p_{i-1}(v_i))\ell_{i, k}\frac{p_{i-1}(v_k)}{1 - X_{i, k}} + p_{i-1}(v_i)\alpha_{v_k}^{1+\gamma}\right) \\
                &\qquad = p_{i-1}(v_k)\log\left(\frac{p_{i-1}(v_k)}{1 - X_{i, k}}\right) \\
                &\qquad \leq p_{i-1}(v_k)\log p_{i-1}(v_k) + (1+\gamma)X_{i, k}p_{i-1}(v_k),
            \end{align*}
            as desired.
        \end{itemize}
        The above covers all cases and completes the proof.
    \end{claimproof}

    It follows from Claim~\ref{claim: qi hypergraph} that
    \[\E[H_i(v_k) \mid \Omega_1, \ldots, \Omega_{i-1}] \geq H_{i-1}(v_k) - (1+\gamma)Q_{i-1}(v_i, v_k).\]
    Noting that $H_0(v_k) = \alpha_{v_k} \log (1/\alpha_{v_k})$ and $d_L(v_k) \leq (r-1)d^*(v_k)$, applying the above recursively along with Lemma~\ref{lemma: expectation energy} completes the proof.
\end{proof}

With the above in hand, let us show that $\alpha_{v_k}^{\gamma}\Pr[v_k \in B_{k-1}^\mu]$ is small.

\begin{lemma}\label{Lemma: bad small hypergraph}
    $\alpha_{v_k}^{\gamma}\Pr[v_k \in B_{k-1}^\mu] \leq \alpha_{v_k}\left(\frac{1}{r} + \gamma\right)$.
\end{lemma}

\begin{proof}
    We have the following:
    \begin{align*}
        -H_{k-1}(v_k) &= p_{k-1}(v_k) \log p_{k-1}(v_k) \\
        &= \left(p_{k-1}(v_k) \log p_{k-1}(v_k) - p_{k-1}(v_k) \log \left(\alpha_{v_k}^{1+\gamma}\right) +  p_{k-1}(v_k) \log \left(\alpha_{v_k}^{1+\gamma}\right)\right) \\
        &= p_{k-1}(v_k)\log \left(\alpha_{v_k}^{1+\gamma}\right) + p_{k-1}(v_k) \log \left(p_{k-1}(v_k)/\alpha_{v_k}^{1+\gamma}\right).
    \end{align*}
    Note that if $p_{k-1}(v_k) > 0$, then $p_{k-1}(v_k) \geq \alpha_{v_k}^{1+\gamma}$ by our lower bound threshold.
    This implies that $p_{k-1}(v_k) \log \left(p_{k-1}(v_k)/\alpha_{v_k}^{1+\gamma}\right) \geq 0$.
    It follows that
    \begin{align*}
        -\E[H_{k-1}(v_k)] &= \E\left[p_{k-1}(v_k)\log \left(\alpha_{v_k}^{1+\gamma}\right) + p_{k-1}(v_k) \log \left(p_{k-1}(v_k)/\alpha_{v_k}^{1+\gamma}\right)\right] \\
        &= \E[p_{k-1}(v_k)]\log \left(\alpha_{v_k}^{1+\gamma}\right) + \E\left[p_{k-1}(v_k) \log \left(p_{k-1}(v_k)/\alpha_{v_k}^{1+\gamma}\right)\right] \\
        &\geq \alpha_{v_k}\log \left(\alpha_{v_k}^{1+\gamma}\right) + \alpha_{v_k}^{\gamma} \log\left(\alpha_{v_k}^{\gamma} / \alpha_{v_k}^{1+\gamma}\right)\Pr[v_k \in B_{k-1}^\mu].
    \end{align*}
    By Lemma~\ref{Lemma: expectation entropy}, we have
    \begin{align}\label{eq: B bound hypergraph}
        \alpha_{v_k}^{\gamma} \log\left(1/ \alpha_{v_k}\right)\Pr[v_k \in B_{k-1}^\mu] &\leq (1+ \gamma)(r-1)d^*(v_k)\alpha_{v_k}^r + \alpha_{v_k}\left(\log \left(1/\alpha_{v_k}^{1+\gamma}\right) - \log \left(1/\alpha_{v_k}\right)\right) \nonumber \\
        &= \alpha_{v_k}\left((1+ \gamma)(r-1)d^*(v_k)\alpha_{v_k}^{r-1} + \gamma \log \left(1/\alpha_{v_k}\right)\right).
    \end{align}
    Noting that the function $x/\log \left(1/x\right)$ is increasing on $[0, 1]$ and
    \[\alpha_{v_k} = (1+\eps/10)\frac{r}{r-1}\phi(v_k) \leq (1+\eps/10)(1-\eps)\left(\frac{\log d^*(v_k)}{r(r-1)^2d^*(v_k)}\right)^{\frac{1}{r-1}},\]
    we have
    \begin{align}
        \frac{(1+\gamma)(r-1)d^*(v_k)\alpha_{v_k}^{r-1}}{\log \left(1/\alpha_{v_k}\right)} &\leq \frac{(1+\gamma)(1-\eps/2)^r\log d^*(v_k)}{r(r-1)\log \left(1/\alpha_{v_k}\right)} \leq \frac{1}{r}.
    \end{align}
    Plugging this into \eqref{eq: B bound hypergraph}, we obtain
    \begin{align*}
         \alpha_{v_k}^{\gamma} \Pr[v_k \in B_{k-1}^\mu] &\leq  \frac{\alpha_{v_k}}{\log \left(1/\alpha_{v_k}\right)}\left((r-1)d\alpha_{v_k}^{r-1} + \gamma\log \left(1/\alpha_{v_k}\right)\right) \\
         &\leq  \alpha_{v_k}\left(\frac{1}{r} + \gamma\right),
    \end{align*}
    as desired.
\end{proof}

Combining \eqref{eq: hypergraph main bound} with Lemma~\ref{Lemma: bad small hypergraph}, we have
\begin{align*}
    \Pr[v_k \in I] &= \E[p_{k-1}(v_k)] - \alpha_{v_k}^{\gamma} \Pr[v_k \in B_{k-1}^\mu] - \alpha_{v_k}^{1 + \gamma}\Pr[v_k \in B_{k-1}^{\ell}] \\
    &\geq \alpha_{v_k} - \alpha_{v_k}\left(\frac{1}{r} + \gamma\right) - \alpha_{v_k}^{1 + \gamma} \\
    &\geq \frac{1}{1 + \eps/10}\left(1 - \frac{1}{r}\right)\alpha_{v_k} = \phi(v_k),
\end{align*}
completing the proof of Proposition~\ref{prop: hypergraph}.

\section{Concluding Remarks}\label{section: concluding remarks}

We conclude with several potential avenues for future work.
In this paper, we introduced a new entropy-based framework for fractional colorings of $d$-degenerate graphs and hypergraphs. 
At its core, the central technical novelty lies in the design and analysis of the independent set sampling algorithm: it iteratively processes vertices according to a well-chosen vertex ordering, adding each vertex $v_i$ to the independent set $I$ with probability proportional to its current weight, and updating the parameters of its right-neighbors while carefully tracking their downstream impact on $I$. 
We believe this procedure provides a versatile framework for tackling broader problems on $d$-degenerate hypergraphs.

In follow-up work to~\cite{Joh_triangle}, Johansson showed that $K_r$-free graphs $G$ of maximum degree $\Delta$ have chromatic number $O_r\left(\frac{\Delta \log\log \Delta}{\log \Delta}\right)$~\cite{J96-Kr}, extending a seminal result of Shearer~\cite{shearer1995independence} on the independence number of such graphs (see~\cite{Molloy, bernshteyn2019johansson, dhawan2025bounds} for recent proofs with improved leading constants; see also~\cite{bonamy2022bounding, DKPS} for an asymptotic improvement when $r = \omega\left(\frac{\log \Delta}{\log \log \Delta}\right)$).
Johansson employed an entropy-based argument to prove this result (see~\cite[\S9]{bansal2015lov} for a simplified exposition).
Unfortunately, an analogous bound does not hold in general for $\chi_f(\cdot)$ when replacing maximum degree $\Delta$ with degeneracy $d$.
Indeed, Brada\v{c}, Fox, Steiner, Sudakov, and Zhang recently established the following lower bound:

\begin{theorem}[\cite{bradavc2026coloring}]
    For an integer $r\geq 3$, there exists $d_0\in \mathbb{N}$ such that for all $d\geq d_0$, there exists a $d$-degenerate $K_r$-free graph $G$ satisfying 
    \[ \chi_f(G) = \Omega_r\left(\frac{d}{\log^{(r-2)}d}\right). \]
\end{theorem}

Here, $\log^{(b)}(\cdot)$ denotes the $b$-fold iterated logarithm defined as $\log^{(b)}(x) \coloneqq \log(\log^{(b-1)}(x))$ with $\log^{(1)}(x) = \log x$.
They ask whether $d$-degenerate $K_r$-free graphs $G$ satisfy $\chi_f(G) = o_r(d)$ for $r\geq 4$ in general~\cite[Problem~6.4]{bradavc2026coloring}.
It is possible that the techniques introduced in this paper could yield an affirmative answer to their question.

\begin{question}
    Can the entropy-based techniques developed in this paper be used to prove that every $d$-degenerate $K_r$-free graph $G$ satisfies $\chi_f(G) = o_r(d)$ for all $r\geq 4$?
\end{question}

The primary obstacle in adapting Johansson's argument to our setting is its heavy reliance on controlling the neighborhood size $|N_R(v)|$ at each vertex $v$.
Because $|N_R(v)|$ cannot be directly controlled in $d$-degenerate graphs, extending this method is non-trivial.
However, this difficulty disappears in the setting of local demands. 
In upcoming work with Davies, Gallardo, Methuku, Nguyen, Seo, and Spiro~\cite{local_shearer_Kr}, we prove a local fractional coloring variant of Johansson's result by combining the ideas of this paper with those of~\cite{bansal2015lov, J96-Kr}, along with several additional technical innovations.

Martinsson and Steiner conjectured that the constant factor in Theorem~\ref{theo: martinsson steiner} can be improved from $4$ to $1$, and that this bound is optimal. 
In subsequent work with Allen and Noel~\cite{high_girth}, we build upon the framework introduced here to confirm their upper bound conjecture for graphs of girth at least $5$ and settle their lower bound conjecture in full.
Specifically, we show that $d$-degenerate girth-$5$ graphs satisfy $\chi_f(G) \le (1+o_d(1))\frac{d}{\log d}$, while there exist $d$-degenerate $K_3$-free graphs with $\chi_f(G) \ge (1-o_d(1))\frac{d}{\log d}$.
The constant $1$ matches the conjectured computational threshold for finding large independent sets in sparse random graphs.
For $r$-uniform hypergraphs, the analogous threshold constant is $(r-1)^{\frac{1}{r-1}}$ (see~\cite[\S2]{triangle_free} for an informal justification of this barrier).
We conjecture that the corresponding upper bound holds for $r$-uniform hypergraphs with girth at least $4$.

\begin{conjecture}
    Let $H$ be an $r$-uniform $d$-degenerate hypergraph with girth at least $4$.
    Then 
    \[ \chi_f(H) \leq (1+o(1))\left((r-1)\frac{d}{\log d}\right)^{\frac{1}{r-1}}. \]
\end{conjecture}

In forthcoming work with Li, Luo, Methuku, Vo, and Yuan~\cite{uncrowded}, we verify this conjecture for uncrowded hypergraphs (i.e., $r$-uniform hypergraphs with girth at least $5$), generalizing a recent result of Methuku, Vo, and the author~\cite{triangle_free}.
To achieve this, we tailor and simplify Algorithm~\ref{algorithm: fcp hypergraph} specifically for the uncrowded setting, yielding both the predicted upper bound and a version formulated in terms of local demands, resolving a recent conjecture of Yu and Zhang~\cite[Conjecture~4.3]{yu2026independent}.

Following work by Frieze and Mubayi~\cite{frieze2013coloring}, Cooper and Mubayi studied the chromatic number of $3$-uniform triangle-free hypergraphs~\cite{cooper2015list}; here, a triangle in a hypergraph is a collection of three edges $e, f, g$ and three vertices $u, v, w$ such that $u \in e\cap f$, $v \in f\cap g$, $w \in e\cap g$, and $\{u, v, w\}\cap e \cap f \cap g = \emptyset$.
They showed that $3$-uniform triangle-free hypergraphs $H$ of maximum degree $\Delta$ satisfy $\chi(H) = O\left(\sqrt{\frac{\Delta}{\log \Delta}}\right)$.
Li and Postle~\cite{li2022chromatic} extended this result to all uniformities, showing that $r$-uniform triangle-free hypergraphs $H$ of maximum degree $\Delta$ satisfy $\chi(H) \leq c_r\left(\frac{\Delta}{\log \Delta}\right)^{\frac{1}{r-1}}$ for some constant $c_r > 0$.\footnote{They prove a more general result for hypergraphs of rank $r$, where each edge contains at most $r$ vertices.}
We conjecture that the same bound holds for fractional coloring when maximum degree $\Delta$ is replaced by degeneracy $d$.

\begin{conjecture}\label{conj: hypergraph}
    Let $H$ be an $r$-uniform $d$-degenerate triangle-free hypergraph for $r \geq 3$.
    Then 
    \[ \chi_f(H) \leq c_r\left(\frac{d}{\log d}\right)^{\frac{1}{r-1}} \]
    for some constant $c_r > 0$.
\end{conjecture}

The parameter update procedure in our proof of Theorem~\ref{theo: hypergraph free ordinary coloring} relies fundamentally on the property that any pair of vertices $v_i, v_k$ belongs to at most one edge.
Because triangle-free hypergraphs need not be linear, resolving Conjecture~\ref{conj: hypergraph} will likely require substantial new ideas.

\vspace{2mm}
\subsection*{Acknowledgments}

We thank Peter Bradshaw, Minh-Quan Vo, and Jing Yu for helpful discussions, as well as Anton Bernshteyn and Abhishek Methuku for valuable comments on an earlier version of this manuscript. 
We are grateful to Peter Allen and Jonathan A. Noel for a suggestion that streamlined our proofs, and to Ewan Davies for drawing our attention to the local fractional coloring implications of our technique.

\vspace{2mm}
\printbibliography

@inproceedings{bansal2015lov,
  title={On the Lov{\'a}sz theta function for independent sets in sparse graphs},
  author={Bansal, Nikhil and Gupta, Anupam and Guruganesh, Guru},
  booktitle={Proceedings of the forty-seventh annual ACM symposium on Theory of Computing},
  pages={193--200},
  year={2015},
  addendum = {Full version: \url{https://arxiv.org/abs/1504.04767}},
}

@book {MolloyReed,
    AUTHOR = {Molloy, M. and Reed, B.},
     TITLE = {Graph Colouring and the Probabilistic Method},
    SERIES = {Algorithms and Combinatorics},
    VOLUME = {23},
 PUBLISHER = {Springer-Verlag, Berlin},
      YEAR = {2002},
     PAGES = {xiv+326},
      ISBN = {3-540-42139-4},
   MRCLASS = {05-02 (05C15 05C80 60-02 60C05)},
  MRNUMBER = {1869439},
MRREVIEWER = {P.\ Mark\ Kayll},
       DOI = {10.1007/978-3-642-04016-0},
       URL = {https://doi.org/10.1007/978-3-642-04016-0},
}

@article{wood2014hypergraph,
  title={Hypergraph colouring and degeneracy},
  author={Wood, David R},
  journal={AUSTRALASIAN JOURNAL OF COMBINATORICS},
  volume={60},
  number={2},
  pages={142--145},
  year={2014}
}

@incollection{KangKelly2023nibble,
  author    = {Dong Yeap Kang and Tom Kelly and Daniela Kühn and Abhishek Methuku and Deryk Osthus},
  title     = {Graph and hypergraph colouring via nibble methods: A survey},
  booktitle = {Proceedings of the 8th European Congress of Mathematics},
  pages     = {771--823},
  year      = {2023},
  publisher = {EMS Press},
  doi       = {10.4171/8ECM/11}
}

@inproceedings{brooks1941colouring,
  title={On colouring the nodes of a network},
  author={R. L. Brooks},
  booktitle={Mathematical Proceedings of the Cambridge Philosophical Society},
  volume={37},
  number={2},
  pages={194--197},
  year={1941},
  organization={Cambridge University Press}
}

@article{PS15,
    author = {S. Pettie and H.-H. Su},
    title = {Distributed coloring algorithms for triangle-free graphs},
    journaltitle = {Information and Computation},
    date = {2015},
    volume = {243},
    pages = {263--280},
}

@article{anderson2025coloring,
  title={Coloring graphs with forbidden almost bipartite subgraphs},
  author={Anderson, James and Bernshteyn, Anton and Dhawan, Abhishek},
  journal={Random Structures \& Algorithms},
  volume={66},
  number={4},
  pages={e70012},
  year={2025},
  publisher={Wiley Online Library}
}

@article{alon1996independence,
  title={Independence numbers of locally sparse graphs and a {R}amsey type problem},
  author={Alon, Noga},
  fjournal={Random Structures \& Algorithms},
  journal   = {Random Structures Algorithms},
  volume={9},
  number={3},
  pages={271--278},
  year={1996},
  publisher={Wiley Online Library}
}

@article{anderson2024coloring,
  title={Coloring locally sparse graphs},
  author={Anderson, James and Dhawan, Abhishek and Kuchukova, Aiya},
  journal={The Electronic Journal of Combinatorics},
  pages={P1.31},
  year={2026}
}

@article{kostochka1999properties,
  title={Properties of Descartes' construction of triangle-free graphs with high chromatic number},
  author={Kostochka, Alexandr V and Ne{\v{s}}et{\v{r}}il, Jaroslav},
  journal={Combinatorics, Probability and Computing},
  volume={8},
  number={5},
  pages={467--472},
  year={1999},
  publisher={Cambridge University Press}
}

@article{descartes1954solution,
  title={Solution to advanced problem no. 4526},
  author={Descartes, Blanche},
  journal={Amer. Math. Monthly},
  volume={61},
  number={352},
  pages={216},
  year={1954}
}

@article{verstraete2026independent,
  title={Independent sets in hypergraphs},
  author={Verstraete, Jacques and Wilson, Chase},
  journal={Random Structures \& Algorithms},
  volume={68},
  number={1},
  pages={e70047},
  year={2026},
  publisher={Wiley Online Library}
}

@article{cooper2015list,
  title={List coloring triangle-free hypergraphs},
  author={Cooper, Jeff and Mubayi, Dhruv},
  journal={Random Structures \& Algorithms},
  volume={47},
  number={3},
  pages={487--519},
  year={2015},
  publisher={Wiley Online Library}
}

@article{nie2021randomized,
  title={Randomized greedy algorithm for independent sets in regular uniform hypergraphs with large girth},
  author={Nie, Jiaxi and Verstra{\"e}te, Jacques},
  journal={Random Structures \& Algorithms},
  volume={59},
  number={1},
  pages={79--95},
  year={2021},
  publisher={Wiley Online Library}
}

@unpublished{li2022chromatic,
  title={The chromatic number of triangle-free hypergraphs},
  author={L. Li and L. Postle},
howpublished = {\url{https://arxiv.org/abs/2202.02839} (preprint)},
  year={2022}
}

@techreport{caro1979new,
  title={New results on the independence number},
  author={Caro, Yair},
  year={1979},
  institution={Technical Report, Tel-Aviv University}
}

@unpublished{triangle_free,
  title={The independence number of uncrowded hypergraphs: bounds matching the shattering threshold},
  author={Abhishek Dhawan and Abhishek Methuku and Minh-Quan Vo},
howpublished = {\url{https://arxiv.org/abs/2606.18048v1} (preprint)},
  year={2026}
}

@unpublished{high_girth,
  title={Sharp bounds for the fractional chromatic number of high-girth $d$-degenerate graphs},
  author={Peter Allen and Abhishek Dhawan and Jonathan A. Noel},
howpublished = {\url{https://arxiv.org/abs/2607.26271} (preprint)},
  year={2026}
}

@article{caro1991improved,
  title={Improved lower bounds on k-independence},
  author={Caro, Yair and Tuza, Zsolt},
  journal={Journal of Graph Theory},
  volume={15},
  number={1},
  pages={99--107},
  year={1991},
  publisher={Wiley Online Library}
}

@article{dutta2012new,
  title={New lower bounds for the independence number of sparse graphs and hypergraphs},
  author={Dutta, Kunal and Mubayi, Dhruv and Subramanian, CR},
  journal={SIAM Journal on Discrete Mathematics},
  volume={26},
  number={3},
  pages={1134--1147},
  year={2012},
  publisher={SIAM}
}

@article{spencer1972turan,
  author = {Spencer, Joel},
  title = {{T}ur{\'a}n's theorem for $k$-graphs},
  journal = {Discrete Mathematics},
  volume = {2},
  number = {2},
  pages = {183--186},
  year = {1972},
}

@article{dhawan2025list,
  author = {Dhawan, Abhishek},
  title = {List colorings of $k$-partite $k$-graphs},
  journal = {The Electronic Journal of Combinatorics},
  volume = {32},
  number = {2},
  pages = {P2.16},
  year = {2025},
}

@unpublished{yu2026independent,
  title={On independent sets in uncrowded uniform hypergraphs},
  author={Yu, Jing and Zhang, Junchi},
  howpublished = {\url{https://arxiv.org/abs/2606.18171} (preprint)},
  year={2026}
}

@unpublished{uncrowded,
  title={Random independent sets in uncrowded hypergraphs},
  author={Abhishek Dhawan and Lina Li and Abhishek Methuku and Minh-Quan Vo and Kewen Yuan},
howpublished = {In preparation}
}

@unpublished{local_shearer_Kr,
  title={Fractionally coloring $K_r$-free graphs with local demands},
  author={Ewan Davies and Abhishek Dhawan and David Gallardo and Abhishek Methuku and Huy Nguyen and Jaehyeon Seo and Samuel Spiro},
howpublished = {In preparation}
}

@article{Vizing,
	author = {V.G. Vizing},
	title = {On an estimate of the chromatic class of a $p$-graph},
	journaltitle = {Diskret. Analiz.},
	volume = {3},
	date = {1964},
	pages = {25--30},
}

@unpublished{Kttt,
  title={Independent sets and colorings of $K_{t, t, t}$-free graphs},
  author={Dhawan, Abhishek and Janzer, Oliver and Methuku, Abhishek},
  howpublished = {\url{https://arxiv.org/abs/2511.17191} (preprint)},
  year={2025}
}

@article{kelly2024fractional,
  title={Fractional coloring with local demands and applications to degree-sequence bounds on the independence number},
  author={Kelly, Tom and Postle, Luke},
  journal={Journal of Combinatorial Theory, Series B},
  volume={169},
  pages={298--337},
  year={2024},
  publisher={Elsevier}
}

@article{harris2019some,
  title={Some results on chromatic number as a function of triangle count},
  author={Harris, David G},
  journal={SIAM Journal on Discrete Mathematics},
  volume={33},
  number={1},
  pages={546--563},
  year={2019},
  publisher={SIAM}
}

@article{ajtai1982extremal,
  title={Extremal uncrowded hypergraphs},
  author={Ajtai, Mikl{\'o}s and Koml{\'o}s, J{\'a}nos and Pintz, Janos and Spencer, Joel and Szemer{\'e}di, Endre},
  journal={Journal of Combinatorial Theory, Series A},
  volume={32},
  number={3},
  pages={321--335},
  year={1982},
  publisher={Elsevier}
}

@article{frieze2013coloring,
  title={Coloring simple hypergraphs},
  author={Frieze, Alan and Mubayi, Dhruv},
  journal={Journal of Combinatorial Theory, Series B},
  volume={103},
  number={6},
  pages={767--794},
  year={2013},
  publisher={Elsevier}
}

@article{duke1995uncrowded,
  title={On uncrowded hypergraphs},
  author={Duke, Richard A and Lefmann, Hanno and R{\"o}dl, Vojtech},
  journal={Random Structures \& Algorithms},
  volume={6},
  number={2-3},
  pages={209--212},
  year={1995},
  publisher={Wiley Online Library}
}

@article{iliopoulos2021improved,
  title={Improved Bounds for Coloring Locally Sparse Hypergraphs},
  author={Iliopoulos, Fotis},
  journal={Approximation, Randomization, and Combinatorial Optimization. Algorithms and Techniques},
  year={2021}
}

@unpublished{bradavc2026coloring,
  title={Coloring small locally sparse degenerate graphs and related problems},
  author={Brada{\v{c}}, Domagoj and Fox, Jacob and Steiner, Raphael and Sudakov, Benny and Zhang, Shengtong},
  howpublished= {\url{https://arxiv.org/abs/2601.15245} (preprint)},
  year={2026}
}

@unpublished{dhawan2024palette,
  title={Palette Sparsification for Graphs with Sparse Neighborhoods},
  author={Dhawan, Abhishek},
  howpublished= {\url{https://arxiv.org/abs/2408.08256} (preprint)},
  year={2024}
}

@article{Molloy,
    author = {M. Molloy},
    title = {The list chromatic number of graphs with small clique number},
    journaltitle = {J. Combin. Theory},
    series = {B},
    volume = {134},
    pages = {264--284},
    date = {2019},
}

@article{AKSConjecture,
    AUTHOR = "N. Alon and M. Krivelevich and B. Sudakov",
    TITLE = "{Coloring graphs with sparse neighborhoods}",
    JOURNAL = "J. Combin. Theory",
    series = {B},
    YEAR = "1999",
    volume = {77},
    pages = {73--82},
}

@article{bernshteyn2019johansson,
  title={The Johansson-Molloy theorem for DP-coloring},
  author={A. Bernshteyn},
  journal={Random Structures \& Algorithms},
  volume={54},
  number={4},
  pages={653--664},
  year={2019},
  publisher={Wiley Online Library}
}

@unpublished{DKPS,
    author = {E. Davies and R.J. Kang and F. Pirot and J.-S. Sereni},
    title = {Graph structure via local occupancy},
    howpublished = {\url{https://arxiv.org/abs/2003.14361} (preprint)},
    date = {2020},
}

@report{Joh_triangle,
	author = {A. Johansson},
	title = {Asymptotic choice number for triangle free graphs},
	type = {Technical Report 91--95},
	institution = {DIMACS},
	date = {1996},
}

@article{davies2018average,
  title={On the average size of independent sets in triangle-free graphs},
  author={Davies, Ewan and Jenssen, Matthew and Perkins, Will and Roberts, Barnaby},
  journal={Proceedings of the American Mathematical Society},
  volume={146},
  number={1},
  pages={111--124},
  year={2018}
}

@article{shearer1995independence,
  title={On the independence number of sparse graphs},
  author={Shearer, James B},
  journal={Random Structures \& Algorithms},
  volume={7},
  number={3},
  pages={269--271},
  year={1995},
  publisher={Wiley Online Library}
}

@article{dhawan2025bounds,
  title={Bounds for the Independence and Chromatic Numbers of Locally Sparse Graphs},
  author={Dhawan, Abhishek},
  journal={Annals of Combinatorics},
  pages={1--28},
  year={2025},
  publisher={Springer}
}

@unpublished{J96-Kr,
author = {Johansson, Anders},
title = {The choice number of sparse graphs},
note = {Unpublished Manuscript},
year = {1996}
}

@article{AndersonBernshteynDhawan,
  title={Coloring graphs with forbidden bipartite subgraphs},
  author={Anderson, James and Bernshteyn, Anton and Dhawan, Abhishek},
  journal={Combinatorics, Probability and Computing},
  volume={32},
  number={1},
  pages={45--67},
  year={2023},
  publisher={Cambridge University Press}
}

@article{frieze1990independence,
  title={On the independence number of random graphs},
  author={Frieze, Alan M},
  journal={Discrete Mathematics},
  volume={81},
  number={2},
  pages={171--175},
  year={1990},
  publisher={Elsevier}
}

@article{bonamy2022bounding,
  title={Bounding $\chi$ by a fraction of $\Delta$ for graphs without large cliques},
  author={Bonamy, Marthe and Kelly, Tom and Nelson, Peter and Postle, Luke},
  journal={Journal of Combinatorial Theory, Series B},
  volume={157},
  pages={263--282},
  year={2022},
  publisher={Elsevier}
}

@inproceedings{martinsson2025random,
  title={Random independent sets in triangle-free graphs},
  author={Martinsson, Anders and Steiner, Raphael},
  booktitle={Forum of Mathematics, Sigma},
  volume={13},
  pages={e156},
  year={2025},
  organization={Cambridge University Press}
}

@article{shearer1991note,
  title={A note on the independence number of triangle-free graphs, II},
  author={Shearer, James B},
  journal={Journal of Combinatorial Theory, Series B},
  volume={53},
  number={2},
  pages={300--307},
  year={1991},
  publisher={Elsevier}
}

\end{document}